\documentclass[12pt,a4paper]{article}
\usepackage{amsmath,amstext,amssymb,amscd, color}
\usepackage{graphicx}

\newcommand{\Xcomment}[1]{}

\newtheorem{theorem}{Theorem}[section]
\newtheorem{lemma}[theorem]{Lemma}
\newtheorem{corollary}[theorem]{Corollary}

\newtheorem{prop}[theorem]{Proposition}

\newcommand{\SEC}[1]{\ref{sec:#1}}  
\newcommand{\SSEC}[1]{\ref{ssec:#1}}  

\makeatletter \@addtoreset{equation}{section} \makeatother

\newenvironment{proof}{\noindent{\bf Proof}~}%
{\hfill$\qed$\medskip}

\def\qed{ \ \vrule width.1cm height.3cm depth0cm}

\newenvironment{numitem1}{\refstepcounter{equation}\begin{enumerate}%
\item[(\thesection.\arabic{equation})]}{\end{enumerate}}

\newcommand{\refeq}[1]{(\ref{eq:#1})}  

 \makeatletter
\renewcommand{\section}{\@startsection{section}{1}{0pt}%
{-3.5ex plus -1ex minus -.2ex}{2.3ex plus .2ex}%
{\normalfont\Large}}
 \makeatother

 \makeatletter
\renewcommand{\subsection}{\@startsection{subsection}{2}{0pt}%
{-3.0ex plus -1ex minus -.2ex}{1.5ex plus .2ex}%
{\normalfont\normalsize\bf}}
 \makeatother

 \makeatletter
\renewcommand{\subsubsection}{\@startsection{subsubsection}{2}{0pt}%
{-2.0ex plus -1ex minus -.2ex}{-2.0ex plus .2ex}%
{\normalfont\normalsize\underline}}
 \makeatother

\def\Rset{{\mathbb R}}
\def\Zset{{\mathbb Z}}

\def\Dscr{{\cal D}}
\def\Escr{{\cal E}}
\def\Fscr{{\cal F}}

\def\Iscr{{\cal I}}

\def\Mscr{{\cal M}}

\def\Cbold{{\bf C}}
\def\Qbold{{\bf Q}}
\def\Sbold{{\bf S}}
\def\Wbold{{\bf W}}

\def\Inver{{\rm Inv}}

\def\tilde{\widetilde}
\def\hat{\widehat}

\def\eps{\epsilon}

\def\Tst{T^{\rm st}}
\def\Tant{T^{\rm ant}}
\def\bd{{\rm bd}}

\def\Zfr{Z^{\,\rm fr}}
\def\Zrear{Z^{\,\rm re}}
\def\Zrim{Z^{\,\rm rim}}
\def\Sfr{S^{\,\rm fr}}
\def\Srear{S^{\,\rm re}}
\def\zetafr{\zeta^{\,\rm fr}}
\def\zetarear{\zeta^{\,\rm re}}
\def\taufr{\tau^{\,\rm fr}}
\def\taurear{\tau^{\,\rm re}}
\def\Mfr{M^{\,\rm fr}}
\def\Pifr{\Pi^{\,\rm fr}}
\def\Pirear{\Pi^{\,\rm re}}
\def\Qcon{Q^{\,\rm con}}
\def\Qfragmin{Q^{\bullet-}}
\def\Qfragpl{Q^{\bullet+}}

\def\Qfrag{Q^{\boxtimes}}
\def\Yturn{\Lambda}

\begin{document}

 \title{On interrelations between strongly, weakly and chord separated set-systems
 (a geometric approach)}

 \author{V.I.~Danilov\thanks{Central Institute of Economics and
Mathematics of the RAS, 47, Nakhimovskii Prospect, 117418 Moscow, Russia;
email: danilov@cemi.rssi.ru.}
 \and
A.V.~Karzanov\thanks{Institute for System Analysis at FRC Computer Science and
Control of the RAS, 9, Prospect 60 Let Oktyabrya, 117312 Moscow, Russia;
emails: sasha@cs.isa.ru; akarzanov7@gmail.com. Corresponding author.
}
  \and
G.A.~Koshevoy\thanks{Central Institute of Economics and Mathematics of the RAS,
47, Nakhimovskii Prospect, 117418 Moscow, Russia; email:
koshevoy@cemi.rssi.ru.}
 }

\date{}

 \maketitle

 \begin{quote}
 {\bf Abstract.} \small We consider three types of set-systems that have
interesting applications in algebraic combinatorics and representation theory:
maximal collections of the so-called \emph{strongly separated}, \emph{weakly
separated}, and \emph{chord separated} subsets of a set $[n]=\{1,2,\ldots,n\}$.
These collections are known to admit nice geometric interpretations; namely,
they are bijective, respectively,  to rhombus tilings on the zonogon
$Z(n,2)$, combined tilings on $Z(n,2)$, and fine zonotopal tilings (or
``cubillages'') on the 3-dimensional zonotope $Z(n,3)$. We describe
interrelations between these three types of set-systems in $2^{[n]}$, by
studying interrelations between their geometric models. In particular, we
completely characterize the sets of rhombus and combined tilings properly
embeddable in a fixed cubillage, explain that they form distributive lattices,
 give efficient methods of extending a given rhombus or combined tiling to a
 cubillage, and etc.

 \medskip
{\em Keywords}\,: strongly separated sets, weakly separated sets, chord
separated sets, rhombus tiling, cubillage, higher Bruhat order

\medskip
{\em AMS Subject Classification}\, 05E10, 05B45
 \end{quote}

\baselineskip=15pt
\parskip=2pt

\section{Introduction}  \label{sec:intr}

For a positive integer $n$, the set $\{1,2,\ldots,n\}$ with the usual order is
denoted by $[n]$. For a set $X\subseteq[n]$ of elements $x_1<x_2<\ldots<x_k$,
we may write $x_1x_2\cdots x_k$ for $X$, $\min(X)$ for $x_1$, and $\max(X)$ for
$x_k$ (where $\min(X)=\max(X):=0$ if $X=\emptyset$). We use three
binary relations on the set $2^{[n]}$ of all subsets of $[n]$. Namely, for
subsets $A,B\subseteq[n]$, we write:
  \begin{numitem1}
  \begin{itemize}
\item[(i)]  $A<B$ if $\max(A)<\min(B)$ (\emph{global dominating});
\item[(ii)]
$A\lessdot B$ if $(A-B)<(B-A)$, where $A'-B'$ denotes $\{i'\colon A'\ni
i'\not\in B'\}$ (\emph{global dominating after cancelations});
\item[(iii)]  $A\rhd B$ if $A-B\ne\emptyset$, and $B-A$
can be expressed as a union of nonempty subsets $B',B''$ so that $B'<(A-B)<B''$
(\emph{splitting}).
  \end{itemize}
  \label{eq:2relat}
   \end{numitem1}

We also say that $A$ \emph{surrounds} $B$ if there are no elements $i<j<k$ of
$[n]$ such that $i,k\in B-A$ and $j\in A-B$ (equivalently, if either $A=B$ or
$A\lessdot B$ or $B\lessdot A$ or $B\rhd A$). The above relations are used in
the following notions.
 \medskip

\noindent \textbf{Definitions.} ~Following Leclerc and Zelevinsky~\cite{LZ},
sets $A,B\subseteq[n]$ are called \emph{strongly separated} (from each other)
if $A\lessdot B$ or $B\lessdot A$ or $A=B$, and called \emph{weakly separated}
if either $|A|\le |B|$ and $A$ surrounds $B$, or $|B|\le|A|$ and $B$ surrounds
$A$ (or both). Following terminology of Galashin~\cite{gal}, $A,B\subseteq [n]$
are called \emph{chord separated} if one of $A,B$ surrounds the other
(equivalently, if there are no elements $i<j<k<l$ of $[n]$ such that $i,k$
belong to one, and $j,\ell$ to the other set among $A-B$ and $B-A$).
 \medskip

(The third notion for $A,B$ is justified by the observation that if $n$ points
labeled $1,2,\ldots,n$ are disposed on a circumference $O$, in this order
cyclically, then there exists a chord to $O$ separating $A-B$ from $B-A$.)

Accordingly, a collection (set-system) $\Fscr\subseteq 2^{[n]}$ is called
strongly (resp. weakly, chord) separated if any two of its members are such.
For brevity, we refer to strongly, weakly, and chord separated collections as
\emph{s-}, \emph{w-}, and \emph{c-collections}, respectively. In the hierarchy
of these collections, any s-collection is a w-collection, and any w-collection
is a c-collection, but the converse need not hold. Such collections are
encountered in interesting applications (in particular, w-collections appeared
in~\cite{LZ} in connection with the problem of quasi-commuting flag minors of a
quantum matrix). Also they admit impressive geometric-combinatorial
representations (which will be discussed later).

An important fact is that these three sorts of collections possess the property
of \emph{purity}. More precisely, we say that a domain $\Dscr\subseteq 2^{[n]}$
is \emph{s-pure} (\emph{w-pure}, \emph{c-pure}) if all inclusion-wise maximal
s-collections (resp. w-collections, c-collections) in $\Dscr$ have the same
cardinality, which in this case is called the \emph{s-rank} (resp.
\emph{w-rank}, \emph{c-rank}) of $\Dscr$. We will rely on the following results
on the full domain $2^{[n]}$.

\begin{numitem1} \label{eq:s-pure}
\cite{LZ} ~$2^{[n]}$ is s-pure and its s-rank is equal to $\binom{n}{2}+
\binom{n}{1}+\binom{n}{0}$ ($=\frac12 n(n+1)+1$).
  \end{numitem1}
\begin{numitem1} \label{eq:w-pure}
\cite{DKK2} ~$2^{[n]}$ is w-pure and its w-rank is equal to $\binom{n}{2}+
\binom{n}{1}+\binom{n}{0}$.
  \end{numitem1}
\begin{numitem1} \label{eq:c-pure}
\cite{gal} ~$2^{[n]}$ is c-pure and its c-rank is equal to
$\binom{n}{3}+\binom{n}{2}+ \binom{n}{1}+\binom{n}{0}$.
  \end{numitem1}

(The phenomenon of w-purity has also been established for some other important
domains, see~\cite{OPS,DKK3,OS}; however, those results are beyond the main
stream of our paper.)

As is seen from~\refeq{s-pure}--\refeq{c-pure}, the c-rank of $2^{[n]}$ is at
$O(n)$ times larger that its s- and w-ranks (which are equal), and we address
the following issue: given a maximal c-collection $C\subset 2^{[n]}$, what can
one say about the sets $\Sbold(C)$ and $\Wbold(C)$ of inclusion-wise maximal
s-collections and w-collections, respectively, contained in $C$?

It turns out that a domain $C$ of this sort need not be s-pure or w-pure in
general, as we show by an example in Sect.~\SSEC{example}. Nevertheless, the
sets of s-collections and w-collections contained in $C$ and having the maximal
\emph{size} (equal to $\frac12 n(n+1)+1$), denoted as $\Sbold^\ast(C)$ and $\Wbold^\ast(C)$,
respectively, have nice structural properties, and to present them is just the
main purpose of this paper.

On this way, we are based on the following known geometric-combinatorial
constructions for s-, w-, and c-collections. As it follows from results
in~\cite{LZ}, each maximal s-collection in $2^{[n]}$ corresponds to the vertex
set of a \emph{rhombus tiling} on the $n$-zonogon in the plane, and vice versa.
A somewhat more sophisticated planar structure, namely, the so-called
\emph{combined tilings}, or \emph{combies}, on the $n$-zonogon are shown to
represent the maximal w-collections in $2^{[n]}$, see~\cite{DKK3}. As to the maximal
c-collections, Galashin~\cite{gal} recently showed that they are bijective to
subdivisions of the 3-dimensional zonotope $Z(n,3)$ into parallelotopes. For
brevity, we liberally refer to such subdivisions as \emph{cubillages}, and its
elements (parallelotopes) as \emph{cubes}.

In this paper, we first discuss interrelations between strongly and chord
separated set-systems. A brief outline is as follows.

(a) For a maximal c-collection $C\subset 2^{[n]}$, let $Q=Q(C)$ be its
associated cubillage (where the elements of $C$ correspond to the 0-dimensional
cells, or \emph{vertices}, of $Q$ regarded as a complex). Then for each
$S\in\Sbold^\ast(C)$, its associated rhombus tiling $T(S)$ is viewed (up to a
piecewise linear deformation) as a 2-dimensional subcomplex of $Q$, called an
\emph{s-membrane} in it. Furthermore, these membranes (and therefore the
members of $\Sbold^\ast(C)$) form a distributive lattice with the minimal and
maximal elements to be the ``front side'' $\Zfr$ and ``rear side'' $\Zrear$ of
the boundary subcomplex of $Z(n,3)$, respectively. This lattice is ``dense'',
in the sense that any two s-collections whose s-membranes are neighboring in
the lattice are obtained from each other by a standard \emph{flip}, or
\emph{mutation} (which involves a hexagon, or, in terminology of Leclerc and
Zelevinsky~\cite{LZ}, is performed ``in the presence of six witnesses'').

(b) It is natural to raise a ``converse'' issue: given a maximal s-collection
$S\subset 2^{[n]}$, what can one say about the set $\Cbold(S)$ of maximal
c-collections containing $S$? One can efficiently construct an instance of such
c-collections, by embedding the tiling $T(S)$ (as an s-membrane) into the
``empty'' zonotope $Z(n,3)$ and then by growing, step by step (or cube by
cube), a required cubillage containing $T(S)$. In fact, the set of
cubillages for $\Cbold(S)$ looks like a ``direct product'' of two sets $\Qbold^-$
and $\Qbold^+$, where the former (latter) is formed by partial
cubillages consisting of ``cubes'' filling the volume of $Z(n,3)$
between the surfaces $\Zfr$ and $T(S)$ (resp. between $T(S)$ and $\Zrear$).
  \Xcomment{
We explain that each of $\Qbold^-, \Qbold^+$ is connected by its own system of
3-flips, where a \emph{3-flip} in $Q$ consists in choosing a sub-zonotope $Z'$
isomorphic to $Z(4,3)$ subdivided into four cubes contained in $Q$, and then
replacing this configuration by the other subdivision of $Z'$ into four cubes
(which is unique). Also other properties are discussed.
  }
 \smallskip

A somewhat similar programme is fulfilled for w-collections, and on this way,
we obtain main results of this paper. We consider a maximal c-collection
$C\subset 2^{[n]}$ and cut each cube of the cubillage $Q$ associated with $C$
into two tetrahedra and one octahedron, forming a subdivision of $Z(n,3)$ into
smaller pieces, denoted as $\Qfrag$ and called the \emph{fragmentation} of $Q$.
We show that each combi $K(W)$ associated with a maximal by size w-collection
$W\subset \Wbold^\ast(C)$ is related to a set of 2-dimensional subcomplexes of
$\Qfrag(C)$, called \emph{w-membranes}. Like s-membranes, the set of all
w-membranes of $\Qfrag$ are shown to form a distributive lattice with the
minimal element $\Zfr$ and the maximal element $\Zrear$, and any two
neighboring w-membranes in the lattice are linked by either a \emph{tetrahedral
flip} or an \emph{octahedral flip} (the latter corresponds, for a w-collection,
to a ``mutation in the presence of four witnesses'', in terminology
of~\cite{LZ}). As to  the ``converse direction'', we consider a fixed maximal
w-collection $W\subset 2^{[n]}$ and develop an efficient geometric method to
construct a cubillage containing the combi $K(W)$.
Also additional results on interrelations between s- and w-collections from
one side, and c-collections from the other side are presented.
  \smallskip

This paper is organized as follows. Section~\SEC{backgr} recalls definitions of
rhombus tilings and combined tilings on a zonogon and fine zonotopal tilings
(``cubillages'') on a 3-dimensional zonotope, and reviews their relations to
maximal s-, w-, and c-collections in $2^{[n]}$. Section~\SEC{smembr} starts
with an example of a maximal c-collection in $2^{[n]}$ that is neither s-pure
nor w-pure. Then it introduces s-membranes in a cubillage, discusses their
relation to rhombus tilings, and describes transformations of cubillages on
$Z(n,3)$ to ones on $Z(n-1,3)$ and back, that are needed for further purposes.
Section~\SEC{lattice_s} studies the structure of the set of s-membranes in a
fixed cubillage and, as a consequence, describes the lattice $\Sbold^\ast(C)$.
Section~\SEC{embed_rt} discusses the task of constructing a cubillage
containing one or two prescribed rhombus tilings. Then we start studying
interrelations between maximal w- and c-collections. In Section~\SEC{w-membr}
we introduce w-membranes in the fragmentation $\Qfrag$ of a fixed cubillage
$Q$, explain that they form a lattice, demonstrate a relationship to combined
tilings, and more. The concluding Section~\SEC{embed_combi} is devoted to the
task of extending a given combi to a cubillage, which results in an efficient
algorithm of finding a maximal c-collection in $2^{[n]}$ containing a given
maximal w-collection.


\section{Backgrounds} \label{sec:backgr}

In this section we recall the geometric representations for  s-, w-, and
c-collections that we are going to use. For disjoint subsets $A$ and
$\{a,\ldots,b\}$ of $[n]$, we use the abbreviated notation $Aa\ldots b$ for
$A\cup\{a,\ldots,b\}$, and write $A-c$ for $A-\{c\}$ when $c\in A$.


\subsection{Rhombus tilings}  \label{ssec:rhomb_til}

Let $\Xi=\{\xi_1,\ldots,\xi_n\}$ be a system of $n$ non-colinear vectors in the
upper hyperplane $\Rset\times \Rset_{\ge 0}$ that follow in this order
clockwise around $(0,0)$. The \emph{zonogon} generated by $\Xi$ is the
$2n$-gone that is the Minkowski sum of segments $[0,\xi_i]$, $i=1,\ldots,n$,
i.e., the set
  $$
Z=Z_\Xi:=\{\lambda_1\xi_1+\ldots+ \lambda_n\xi_n\colon \lambda_i\in\Rset,\;
0\le\lambda_i\le 1,\; i=1,\ldots,n\},
  $$
also denoted as $Z(n,2)$. A tiling that we deal with is a subdivision $T$ of
$Z$ into \emph{tiles}, each being a parallelogram of the form $\sum_{k\in X} \xi_k
+\{\lambda\xi_i+\lambda'\xi_j\colon 0\le \lambda,\lambda'\le 1\}$ for
some $i<j$ and some subset $X\subseteq[n]-\{i,j\}$. In other words, the tiles
are not overlapping (have no common interior points) and their union is $Z$. A
tile determined by $X,i,j$ as above is called an $ij$-\emph{tile} and denoted
as $\rho(X|ij)$.

We identify each subset $X\subseteq[n]$ with the point $\sum_{i\in X} \xi_i$ in
$Z$ (assuming that the generators $\xi_i$ are $\Zset$-independent). Depending
on the context, we may think of $T$ as a 2-dimensional complex and associate to
it the planar directed graph $(V_T,E_T)$ in which each vertex (0-dimensional
cell) is labeled by the corresponding subset of $[n]$ and each edge
(1-dimensional cell) that is a parallel translate of $\xi_i$ for some $i$ is
called an $i$-\emph{edge}, or an edge of \emph{type} (or \emph{color}) $i$. In
particular, the \emph{left boundary} of the zonogon is the directed path
$(v_0,e_1,v_1,\ldots,e_n,v_n)$ in which each vertex $v_i$ is the set $[i]$ (and
$e_i$ is an $i$-edge), whereas the \emph{right boundary} of $Z$ is the directed
path $(v'_0,e'_1,v'_1,\ldots, e'_n,v'_n)$ with $v'_i=[n]-[n-i]$ (and $e'_i$
being an $(n-i+1)$-edge).

We call the vertex set $V_T$ (regarded as a set-system in $2^{[n]}$) the
\emph{spectrum} of $T$. In fact, the graphic structure of $T$ (and therefore
its spectrum) does not depend on the choice of generating vectors $\xi_i$ (by
keeping their ordering clockwise). In the literature one often takes vectors
of the same euclidean length, in which case each tile becomes a rhombus and $T$
is called a \emph{rhombus tiling}. In what follows we will liberally use this
term whatever generators $\xi_i$ are chosen.

One easily shows that for any $1\le i<j\le n$, there exists
a unique $ij$-tile, or $ij$-\emph{rhombus}, in $T$. The central property of
rhombus tilings is as follows.

 \begin{theorem} {\rm \cite{LZ}} \label{tm:LZ}
The correspondence $T\mapsto V_T$ gives a bijection between the set ${\bf
RT}_n$ of rhombus tilings on $Z(n,2)$ and the set ${\bf S}_n$ of maximal
s-collections in $2^{[n]}$.
  \end{theorem}

In particular, each maximal s-collection $S$ determines a unique rhombus tiling
$T$ with $V_T=S$, and this $T$ is constructed easily: each pair of vertices of
the form $X,Xi$ is connected by (straight line) edge from $X$ to $Xi$; then the
resulting graph is planar and all its faces are rhombi, giving $T$. Two rhombus
tilings play an especial role. The spectrum of one, called the \emph{standard
tiling} and denoted as $\Tst_n$, is formed by all \emph{intervals} in $[n]$,
i.e., the sets $I_{ij}:=\{i,i+1,\ldots,j\}$ for $1\le i\le j\le n$, plus the
``empty interval'' $\emptyset$. The other one, called the \emph{anti-standard
tiling} and denoted as $\Tant_n$, has the spectrum consisting of all
\emph{co-intervals}, the sets of the form $[n]-I_{ij}$. These two tilings for
$n=4$ are illustrated on the picture.

 \vspace{-0.3cm}
\begin{center}
\includegraphics{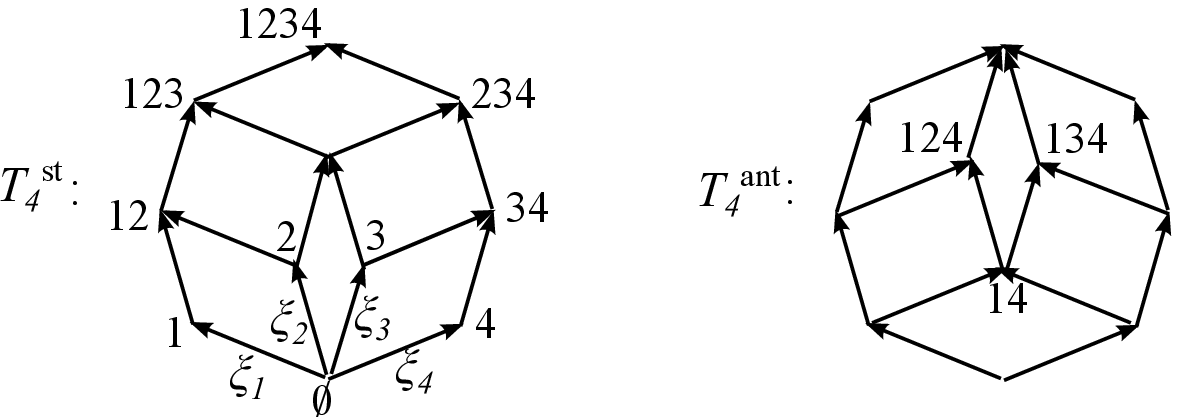}
\end{center}
\vspace{-0.2cm}

Next, as it follows from results in~\cite{LZ}, ${\bf RT}_n$ is endowed with a
poset structure. In this poset, $\Tst_n$ and $\Tant_n$ are the unique minimal
and maximal elements, respectively, and a tiling $T$ immediately precedes a
tiling $T'$ if $T'$ is obtained from $T$ by one \emph{strong} (or hexagonal)
\emph{raising flip} (and accordingly $T$ is obtained from $T'$ by one
\emph{strong lowering flip}). This means that
  \begin{numitem1}  \label{eq:strong_flip}
there exist $i<j<k$ and $X\subseteq [n]-\{i,j,k\}$ such that: $T$ contains the
vertices $X,Xi,Xj,Xk,Xij,Xjk,Xijk$, and the set $V_{T'}$ is obtained from $V_T$ by
replacing $Xj$ by $Xik$.
  \end{numitem1}
(This transformation is called in~\cite{LZ} a ``mutation in the presence of six
witnesses'', namely, $X,Xi,Xk,Xij,Xjk,Xijk$.) See the picture.

\vspace{0cm}
\begin{center}
\includegraphics[scale=0.9]{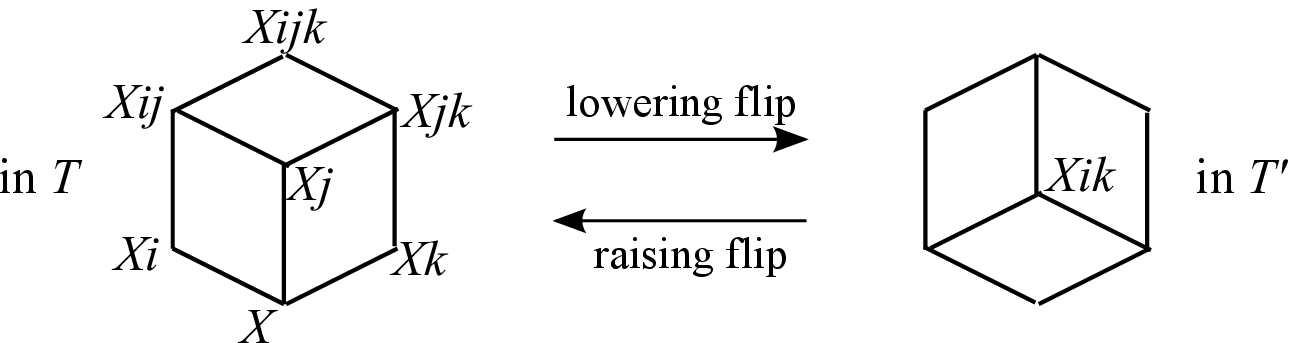}
\end{center}
\vspace{0cm}

We denote the corresponding hexagon in $T$ as $H=H(X|ijk)$ and say that $H$ has
$Y$-\emph{configuration} ($\Yturn$-\emph{configuration}) if the three rhombi
spanning $H$ are as illustrated in the left (resp. right) fragment of the above
picture.


\subsection{Combined tilings}  \label{ssec:combi}

For tilings of this sort, the system $\Xi$ generating the zonogon is required
to satisfy the additional condition of \emph{strict convexity}, namely: for any
$1\le i<j<k\le n$,
  \begin{equation} \label{eq:strict_xi}
\xi_j=\lambda\xi_i+\lambda'\xi_k,\quad \mbox{where $\lambda,\lambda'\in
\Rset_{>0}$ and $\lambda+\lambda'>1$}.
  \end{equation}

Besides, we use vectors $\eps_{ij}:=\xi_j-\xi_i$ for $1\le i<j\le n$. A
\emph{combined tiling}, or simply a \emph{combi}, is a subdivision $K$ of
$Z_\Xi$ into certain polygons specified below. Like the case of rhombus
tilings, a combi $K$ may be regarded as a complex and we associate to it a
planar directed graph $(V_K,E_K)$ in which each vertex corresponds to some
subset of $[n]$ and each edge is a parallel translate of either $\xi_i$ or
$\eps_{ij}$ for some $i,j$. In the later case we say that the edge has
\emph{type} $ij$. We call $V_K$ the \emph{spectrum} of $K$.

There are three sorts of tiles in $T$: $\Delta$-tiles, $\nabla$-tiles, and
lenses. A $\Delta$-\emph{tile} ($\nabla$-\emph{tile}) is a triangle with
vertices $A,B,C\subseteq[n]$ and edges $(B,A),(C,A),(B,C)$ (resp.
$(A,C),(A,B),(B,C)$) of types $i$, $j$ and $ij$, respectively, where $i<j$. For
purposes of Sect.~\SEC{w-membr}, we denote this tile as $\Delta(A|ji)$
(resp. $\nabla(A|ij)$), call $(B,C)$ its \emph{base} edge and call $A$ its
\emph{top} (resp. \emph{bottom}) vertex. See the left and middle fragments of
the picture.

\vspace{-0.2cm}
\begin{center}
\includegraphics{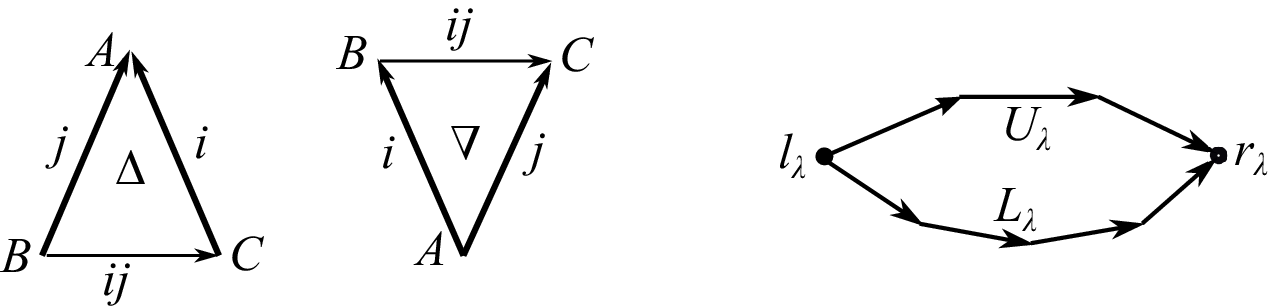}
\end{center}
\vspace{-0.2cm}

In a \emph{lens} $\lambda$, the boundary is formed by two directed paths
$U_\lambda$ and $L_\lambda$, with at least two edges in each, having the same
beginning vertex $\ell_\lambda$ and the same end vertex $r_\lambda$; see the
right fragment of the above picture. The \emph{upper boundary}
$U_\lambda=(v_0,e_1,v_1,\ldots,e_p,v_p)$ is such that $v_0=\ell_\lambda$,
$v_p=r_\lambda$, and $v_k=Xi_k$ for $k=0,\ldots,p$, where $p\ge 2$,
$X\subset[n]$ and $i_0<i_1<\cdots <i_p$ (so $k$-th edge $e_k$ is of type
$i_{k-1}i_k$). And the \emph{lower boundary}
$L_\lambda=(u_0,e'_1,u_1,\ldots,e'_q,u_q)$ is such that $u_0=\ell_\lambda$,
$u_q=r_\lambda$, and $u_m=Y-j_m$ for $m=0,\ldots,q$, where $q\ge 2$,
$Y\subseteq [n]$ and $j_0>j_1>\cdots>j_q$ (so $m$-th edge $e'_m$ is of type
$j_mj_{m-1}$). Then $Y=Xi_0j_0=Xi_pj_q$, implying $i_0=j_q$ and $i_p=j_0$, and
we say that the lens $\lambda$ has \emph{type} $i_0j_0$. Note that $X$ as well
as $Y$ need not be a vertex in $K$; we refer to $X$ and $Y$ as the \emph{lower}
and \emph{upper root} of $\lambda$, respectively. Due to
condition~\refeq{strict_xi}, each lens $\lambda$ is a convex polygon of which
vertices are exactly the vertices of $U_\lambda\cup L_\lambda$.
  \medskip

\noindent\textbf{Remark 1.} In the definition of a combi introduced
in~\cite{DKK3}, the generators $\xi_i$ are assumed to have the same euclidean
length. However, taking arbitrary (cyclically ordered) generators subject
to~\refeq{strict_xi} does not affect, in essence, the structure of the combi
and its spectrum, and in what follows we will vary the choice of generators
when needed. Next, to simplify visualizations, it will be convenient to think
of edges of type $i$ as ``almost vertical'', and of edges of type $ij$ as
``almost horizontal''; following terminology of~\cite{DKK3}, we refer to the
former edges as \emph{V-edges}, and to the latter ones as \emph{H-edges}. Note
that any rhombus tiling turns into a combi without lenses in a natural way:
each rhombus is subdivided into two ``semi-rhombi'' $\Delta$ and $\nabla$ by
drawing the ``almost horizontal'' diagonal in it.
 \smallskip

The picture below illustrates a combi $K$ having one lens $\lambda$ for $n=4$;
here the V-edges and H-edges are drawn by thick and thin lines, respectively.

\vspace{0cm}
\begin{center}
\includegraphics{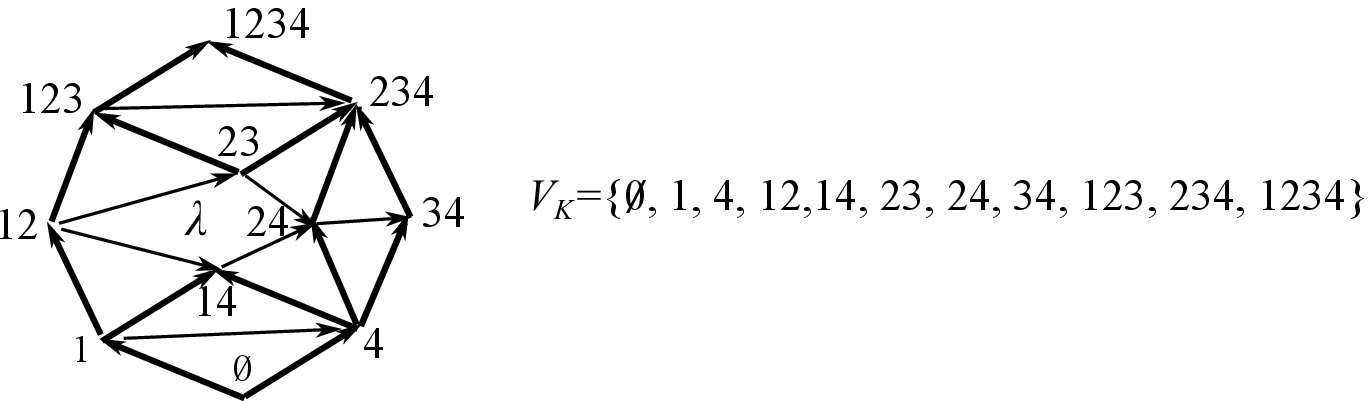}
\end{center}
\vspace{0cm}

We will rely on the following central result on combies.

 \begin{theorem} {\rm\cite{DKK3}} \label{tm:combi}
The correspondence $K\mapsto V_K$ gives a bijection between the set ${\bf K}_n$
of combined tilings on $Z(n,2)$ and the set ${\bf W}_n$ of maximal
w-collections in $2^{[n]}$.
  \end{theorem}

In particular, each maximal w-collection $W$ determines a unique combi $K$ with
$V_K=W$, and~\cite{DKK3} explains how to construct this $K$ efficiently.

By results in~\cite{DKK1,DKK2,DKK3}, the set ${\bf K}_n$ forms a poset in which
$\Tst_n$ and $\Tant_n$ are the unique minimal and maximal elements,
respectively, and a combi $K$ immediately precedes a combi $K'$ if $K'$ is
obtained from $K$ by one \emph{weak raising flip} (and accordingly $K$ is
obtained from $K'$ by one \emph{weak lowering flip}). This means that
  \begin{numitem1} \label{eq:weak_flip}
there exist $i<j<k$ and $X\subseteq [n]-\{i,j,k\}$ such that: $K$ contains the
vertices $Xi,Xj,Xk,Xij,Xjk$, and the set $V_{K'}$ is obtained from $V_K$ by
replacing $Xj$ by $Xik$.
  \end{numitem1}
(Using terminology of~\cite{LZ}, one says that $V_K$ and $V_{K'}$ are linked by a
``mutation in the presence of four witnesses'', namely, $Xi,Xk,Xij,Xjk$.)


\subsection{Cubillages}  \label{ssec:cubil}
Now we deal with the zonotope generated by a special cyclic configuration
$\Theta$ of vectors in the space $\Rset^3$ with coordinates $(x,y,z)$. It
consists of $n$ vectors $\theta_i=(x_i,y_i,1)$, $i=1,\ldots,n$, with the
following strict convexity condition:
  \begin{numitem1} \label{eq:cyc_conf}
$x_1<x_2<\cdots<x_n$, and each $(x_i,y_i)$ is a vertex of the convex hull $H$
of the points $(x_1,y_1),\ldots,(x_n,y_n)$ in the plane $z=1$.
  \end{numitem1}

\noindent An example with $n=5$ is illustrated in the picture (where
$y_i=x_i^2$ and $x_i=-x_{6-i}$).

  \vspace{0cm}
\begin{center}
\includegraphics{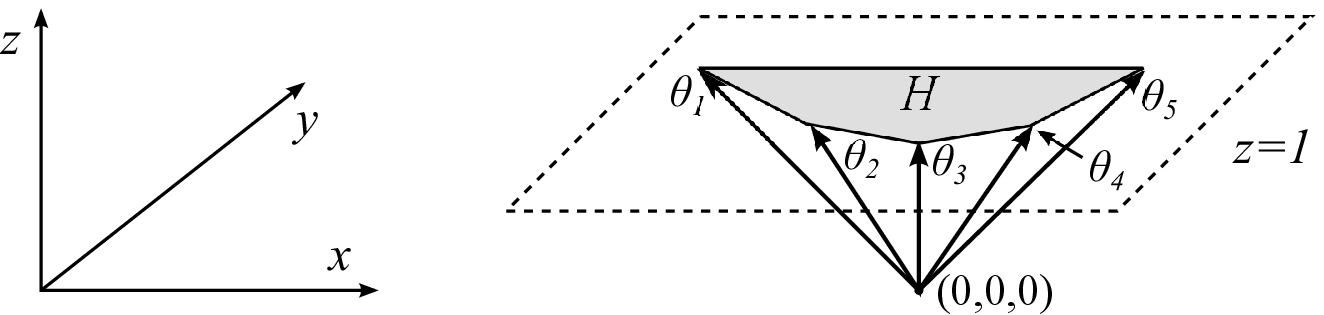}
\end{center}
\vspace{0cm}

The \emph{zonotope} $Z$ generated by $\Theta$, also denoted as $Z(n,3)$, is the
sum of segments $[0,\theta_i]$, $i=1,\ldots,n$. A \emph{fine zonotopal tiling}
of $Z$ is a subdivision  $Q$ of $Z$ into parallelotopes of which any two can
intersect only by a common face and any face of the boundary of $Z$ is a one of
some parallelotope. This is possible only if each parallelotope is of the form
$\sum_{i\in X}\theta_i +\{\lambda\theta_i+\lambda'\theta_j+
\lambda''\theta_k\colon 0\le \lambda,\lambda',\lambda''\le 1\}$ for some
$i<j<k$ and $X\subseteq [n]-\{i,j,k\}$. (For more aspects of fine zonotopal
tilings on zonotopes generated by cyclic configurations, see, e.g.,
\cite{GP}.) For brevity, we liberally refer to parallelotopes as \emph{cubes},
and to $Q$ as a \emph{cubillage}.

Depending on the context, we also may think of a cubillage $Q$ as a polyhedral
complex or as the corresponding set of cubes. In particular, (in the former
case) by a vertex, edge, rhombus in $Q$ we mean, respectively, (the closure of)
a 0-, 1-, 2-dimensional cell of this complex, and (in the latter case) when
writing $\zeta\in Q$, we mean that $\zeta$ is a cube of $Q$.

Like the case of zonogons and rhombus tilings, each subset $X\subseteq [n]$ is
identified with the point $\sum_{i\in X}\theta_i$ in $Z$ (assuming that
generators $\theta_i$ are $\Zset$-independent), and we apply to an edge,
rhombus, and cube in $Q$ terms an $i$-\emph{edge}, $ij$-\emph{rhombus}, and
$ijk$-\emph{cube} in a due way (where $i<j<k$). Also we say that such a rhombus
(cube) is of \emph{type} $ij$ (resp. $ijk$). The edges are directed according
to the generating vectors. An $ij$-rhombus ($ijk$-cube) with the bottom vertex
$X$ is denoted as $\rho(X|ij)$ (resp. $\zeta(X|ijk)$). As a specialization to
$d=3$ of a well-known fact about fine zonotopal tilings on zonotopes $Z(n,d)$
generated by cyclic vector configurations in $\Rset^d$ with an arbitrary
dimension $d$, the following is true:
   \begin{numitem1} \label{eq: one_ijk}
for any $1\le i<j<k\le n$, a cubillage $Q$ has exactly one $ijk$-cube.
  \end{numitem1}

The directed graph formed by the vertices and edges occurring in $Q$ is denoted
by $G_Q=(V_Q,E_Q)$ and we call the vertex set $V_Q$ regarded as a set-system in
$2^{[n]}$ the \emph{spectrum} of $Q$. The following property is of most
importance for us.
 \begin{theorem} {\rm\cite{gal}} \label{tm:galash}
The correspondence $Q\mapsto V_Q$ gives a bijection between the set ${\bf Q}_n$
of cubillages on $Z(n,3)$ and the set ${\bf C}_n$ of maximal c-collections in
$2^{[n]}$.
 \end{theorem}

Next, in our study of interrelations of s- and c-collections, we will use the
projection $\pi:\Rset^3\to\Rset^2$ along the second coordinate vector, i.e.,
given by $\pi(x,y,z):=(x,z)$. Then $\pi(Z)$ is the zonogon generated by the
vectors $\pi(\theta_1),\ldots, \pi(\theta_n)$ (which lie in the ``upper
half-plane'' and are numbered clockwise, in view of~\refeq{cyc_conf}). Let us
represent the boundary $\bd(Z)$ of $Z$ as the union $\Zfr\cup \Zrear$ of its
\emph{front} and \emph{rear} sides, i.e., $\Zfr$ ($\Zrear$) is formed by the
points $(x,y,z)\in Z$ with $y$ minimal (resp. maximal) in $\pi^{-1}(x,z)$. Then
$\Zrim:=\Zfr\cap \Zrear$ is the closed piecewise linear curve being the union
of two directed paths connecting the vertices $\emptyset$
and $[n]$ in $G_Q$. We call $\Zrim$ the \emph{rim} of $Z$.
Condition~\refeq{cyc_conf} provides that
  \begin{numitem1} \label{eq:fr-rear}
the maximal affine sets in $\Zfr$ and $\Zrear$ are rhombi which are projected
by $\pi$ to the standard and antistandard tilings in $\pi(Z)$, respectively
(defined in Sect.~\SSEC{rhomb_til}).
  \end{numitem1}

We identify $\Zfr$ and $\Zrear$ with the corresponding polyhedral complexes.

Finally, for $h=0,1,\ldots,n$, the intersection of $Z$ with the horizontal
plane $z=h$ is denoted as $\Sigma_h$ and called $h$-th \emph{section} of
$Z$; the definition of $\Theta$ implies that $\Sigma_h$ contains all vertices
$X$ of size $|X|=h$ in $Q$.


\section{S-membranes} \label{sec:smembr}

This section starts with an example of cubillages whose spectra are neither
s-pure nor w-pure. Then we consider a fixed cubillage $Q$ on $Z(n,3)$,
introduce a class of 2-dimensional subcomplexes in it, called
\emph{s-membranes}, explain that each of them is isomorphic to a rhombus tiling
$T$ on $Z(n,2)$ such that $V_T\subset V_Q$, and vice versa (thus obtaining a
geometric description of $\Sbold^\ast(V_Q)$), and demonstrate some other
structural properties.


\subsection{An example}  \label{ssec:example}

Consider the zonotope $Z=Z(4,3)$. The vertices of its boundary $\bd(Z)$ are the
intervals and co-intervals on the set $[4]$ (cf.~\refeq{fr-rear}), and there
are exactly two subsets of $[4]$ that are neither intervals nor co-intervals,
namely, $13$ and $24$. So $13$ and $24$ are just those ``points'' from
$2^{[4]}$ that are contained in the interior of $Z$. Since they are not chord
separated, there are exactly two cubillages on $Z$: one containing $13$ and the
other containing $24$ (taking into account that the vertices of $\bd(Z)$ belong
to any cubillage and that each cubillage is determined by its spectrum, by
Theorem~\ref{tm:galash}).
  \begin{lemma} \label{lm:example}
For the cubillage $Q$ on $Z(4,3)$ that contains $13$, the set $V_Q$ is neither
s-pure nor w-pure.
  \end{lemma}
  \begin{proof}
Let $R,V_1,V_2$ be the vertices in the rim, front side, and rear side of
$Z(4,3)$, respectively (for definitions, see the end of Sect.~\SSEC{cubil}).
Then $R$ consists of the eight intervals of the form $[i]$ or $[4]-[i]$ ($0\le
i\le 4$); $V_1$ is $R$ plus the intervals $2,3,23$; and $V_2$ is $R$ plus the
co-intervals $14, 124, 134$. Note that the vertices (intervals) of the rim of
any zonotope $Z(n,3)$ are strongly separated from any subset of $[n]$.

Consider the set $S:=R\cup\{2,124\}$. It is a subset  of
  $$
  V_Q=V_1\cup V_2\cup \{13\}= R\cup\{2,3,23,14,124,134,13\}.
  $$
Observe that $S$ is an s-collection (since $2\subset 124$) but not an
s-collection of maximum size in $2^{[4]}$ (since $|S|=10$ but $|V_1|=11$). We
have $3,23\rhd 124$ but $|3|,|23|<|124|$, and $2\rhd 14,134,13$ but $|2|<
|14|,|134|,|13|$. Thus, $S$ is a maximal s-collection and a maximal
w-collection in $V_Q$, yielding the result.
  \end{proof}

In fact, using results on s- and w-membranes given later, one can strengthen
the above lemma by showing that for $n\ge 4$, the spectrum $V_Q$ of \emph{any}
cubillage $Q$ on $Z(n,3)$ is neither s-pure nor w-pure (we omit a proof here).

\subsection{S-membranes}  \label{ssec:smembr}

Like the definitions of $\Zfr$ and $\Zrear$ from Sect.~\SSEC{cubil}, for
$S\subseteq Z=Z(n,3)$, let $\Sfr$ ($\Srear$) denote the set of points
$(x,y,z)\in S$ with $y$ minimum (resp. maximum) for each $\pi^{-1}(x,z)$,
called the \emph{front} (resp. \emph{rear}) part of $S$. (In other words,
$\Sfr$ and $\Srear$ are what is seen in $S$ in the directions $(0,1,0)$ and
$(0,-1,0)$, respectively.) This is extended in a natural way when we deal with
a subcomplex of a cubillage in $Z$.
 \medskip

\noindent \emph{Example.} In view of~\refeq{cyc_conf}, for a cube
$\zeta=\zeta(X|ijk)$ (where $i<j<k)$), ~$\zetafr$ is formed by the rhombi
$\rho(X|ij), \rho(X|ji), \rho(Xj|ik)$, while $\zetarear$ is formed by
$\rho(X|ik), \rho(Xi|jk), \rho(Xk|ij)$. See the picture.

  \vspace{0cm}
\begin{center}
\includegraphics[scale=0.9]{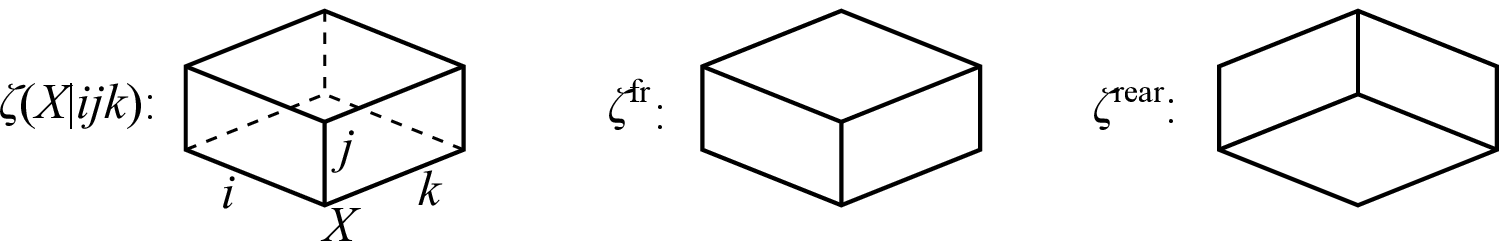}
\end{center}
\vspace{0cm}

\noindent\textbf{Definition.} A 2-dimensional subcomplex $M$ of a cubillage $Q$
is called an \emph{s-membrane} if $\pi$ is injective on $M$ and sends it
to a rhombus tiling on the zonogon $Z(n,2)$. In other words, $M$ is a disk
(i.e., a shape homeomorphic to a circle) whose boundary coincides with $\Zrim$
and such that $M=\Mfr$.
 \medskip

In particular, both $\Zfr$ and $\Zrear$ are s-membranes. Therefore, up to a
piecewise linear deformation, we may think of $M$ as a rhombus tiling whose
spectrum is contained in $V_Q$. So the vertex set $V_M$ of $M$ belongs to ${\bf
S}^\ast(V_Q)$. Moreover, a sharper property takes place (which can be deduced
from general results on higher Bruhat orders and related aspects
in~\cite{MS,VK,zieg}; yet we prefer to give a direct and shorter proof).

\begin{theorem} \label{tm:smembr-scoll}
The correspondence $M\mapsto V_M$ gives a bijection between the s-membranes in
a cubillage $Q$ on $Z(n,3)$ and the set ${\bf S}^\ast(V_Q)$ of maximum by size
s-collections contained in $V_Q$.
 \end{theorem}

In light of explanations above, it suffices to prove the following
  \begin{prop} \label{pr:T-smembr}
For any rhombus tiling $T$ on $Z(n,2)$ with $V_T\subset V_Q$, there exists an
s-membrane $M$ in $Q$ isomorphic to $T$.
  \end{prop}

This proposition will be proved in Sect.~\SSEC{appl_contr}, based on a more
detailed study of structural features of cubillages and operations on them
given in the next subsection.

\subsection{Pies in a cubillage}  \label{ssec:pies}

Given a cubillage $Q$ on $Z=Z(n,3)$, let $\Pi_i=\Pi_i(Q)$ be the part of $Z$
covered by cubes of $Q$ having edges of color $i$, or, let us say,
$i$-\emph{cubes}. When it is not confusing, we also think of $\Pi_i$ as the set
of $i$-cubes or as the corresponding subcomplex of $Q$. We refer to $\Pi_i$ as
$i$-th \emph{pie} of $Q$. When $i=n$ or 1, the pie structure becomes rather
transparent, which will enable us to apply some useful reductions.

To clarify the structure of $\Pi_n$, we first consider the set $U$ of $n$-edges
lying in $\bd(Z)$. Since the tilings on the sides $\Zfr$ and $\Zrear$ of $Z$
are isomorphic to $\Tst_n$ and $\Tant_n$, respectively (cf.~\refeq{fr-rear}),
one can see that
   \begin{numitem1} \label{eq:C-Cp}
the beginning vertices of edges of $U$ are precisely those contained in the
cycle $C=P'\cup P''$, where $P'$ is the subpath of left path of $\Zrim$ from
the bottom vertex $\emptyset$ to $[n-1]$, and $P''$ is the path in $\Zfr$
passing the vertices $[n-1]-[i]$ for $i=n-1,n-2,\ldots,0$; in other words, $C$
is the rim of the zonotope $Z(n-1,3)$ generated by
$\theta_1,\ldots,\theta_{n-1}$.
  \end{numitem1}

\noindent Accordingly, the end vertices of edges of $U$ lie on the cycle
$C':=C+\theta_n$; this $C'$ is viewed as the rim of the zonotope $Z(n-1,3)$
shifted by $\theta_n$. The area of $\bd(Z)$ between $C$ and $C'$ is subdivided
into $2(n-2)$ rhombi of types $\ast n$ (where $\ast$ means an element of
$[n-1]$); we call this subdivision the \emph{belt} of $\Pi_n$. See the picture
with $n=5$.

  \vspace{0cm}
\begin{center}
\includegraphics[scale=1]{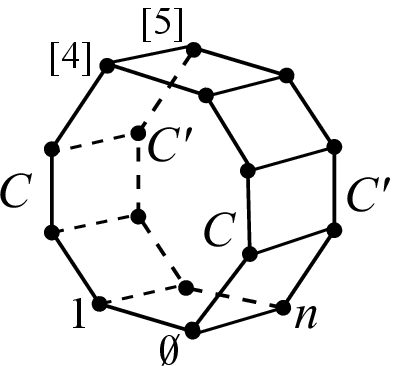}
\end{center}
\vspace{0cm}

Now fix an $n$-edge $e=(X,Xn)$ not on $\bd(Z)$ and consider the set $S$ of
cubes in $\Pi_n$ containing $e$. Each cube $\zeta\in S$ is viewed as the
(Minkowski) sum of some rhombus $\rho=\rho(X|ij)$ and the segment
$[0,\theta_n]$, and an important fact is that $\rho$ belongs to the front side
of $\zeta$ (in view of~\refeq{cyc_conf} and $n>i,j$). Gluing together such
rhombi $\rho$, we obtain a disk lying on the front side of the shape $\hat
S:=\cup(\zeta\in S)$ and containing $X$ as an interior point, and $\hat S$ is
just the sum of this disk and $[0,\theta_n]$. Based on this local behavior, one
can conclude that
  \begin{numitem1} \label{eq:Pin}
$\Pi_n$ is the sum of a disk $D$ and the segment $[0,\theta_n]$; this disk lies
in $\Pifr_n$ and its boundary is formed by the cycle $C$ as in~\refeq{C-Cp}.
  \end{numitem1}

Then $D':=D+\theta_i$ is a disk in $\Pirear_n$ whose boundary is the cycle $C'$
as above.

The facts that $D^{\rm fr}=D$ and $C=\Zrim(n-1,3)$ imply that $D$ is subdivided
into rhombi which (being projected by $\pi$) form a rhombus tiling on
$Z(n-1,2)$. And similarly for $D'$.

In what follows we write $\Pi^-_n$ for $D$, $\Pi^+_n$ for $D'$, $Z^-_n$
($Z^+_n$) for the (closed) subset of $Z$ between $\Zfr$ and $\Pi^-_n$ (resp.
between $\Pi^+_n$ and $\Zrear$), and $Q^-_n$ ($Q^+_n$) for the portion (\emph{partial
cubillage}) of $Q$ lying in $Z^-_n$ (resp. $Z^+_n$). One can see that
   \begin{numitem1} \label{eq:Z-Z+}
the edges of $G_Q$ connecting $Z^-_n$ and $Z^+_n$ are directed from the former
to the latter and are exactly the $n$-edges of $Q$; as a consequence, each
vertex of $Q^-_n$ is in $[n-1]$ and each vertex of $Q^+_n$ is of the form $Xn$,
where $X\subseteq [n-1]$.
  \end{numitem1}

The following operation is of importance.
\smallskip

\noindent\textbf{$n$-Contraction.} Delete the interior of $\Pi_n$ and shift
$Z^+_n$ together with the cubillage $Q^+_n$ filling it by the vector
$-\theta_n$. As a result, the disks $\Pi^-_n$ and $\Pi^+_n$ merge and we obtain
a cubillage on the zonotope $Z(n-1,3)$; it is denoted by $\Qcon_n$ and called
the \emph{contraction} of $Q$ by (the color) $n$.

Note that $\Pi^-_n$ becomes an s-membrane of $\Qcon_n$. Also the following is
obvious:
   \begin{numitem1} \label{eq:contr}
each cube $\zeta=\zeta(X|ijk)$ of $Q$ with $k<n$ (i.e. not contained in
$\Pi_n$) one-to-one corresponds to a cube $\zeta'$ of $\Qcon_n$; this $\zeta'$
is of the form $\zeta(X|ijk)$ if $\zeta\in Q^-_n$, and $\zeta(X-n|ijk)$ if
$\zeta\in Q^+_n$.
  \end{numitem1}

Next we introduce a converse operation.
\smallskip

\noindent\textbf{$n$-Expansion.} Let $M$ be an s-membrane in a cubillage $Q'$
on the zonogon $Z'=Z(n-1,3)$. Define $Z^-(M)$ ($Z^+(M)$) to be the part of
$Z'$ between $(Z')^{\rm fr}$ and $M$ (resp. between $M$ and $(Z')^{\rm re}$),
and define $Q^-(M)$ ($Q^+(M)$) to be the subcubillage of $Q'$ contained in
$Z^-(M)$ (resp. $Z^+(M)$). The $n$-\emph{expansion} operation for $(Q,M)$
consists in shifting $Z^+(M)$ together with $Q^+(M)$ by $\theta_n$ and filling
the space ``between'' $M$ and $M+\theta_n$ by the corresponding set of
$n$-cubes, denoted as $Q^0_n(M)$. More precisely, each rhombus $\rho(X|ij)$ in
$M$ (where $i<j<n$ and $X\subset [n-1]-\{i,j\}$) generates the cube
$\zeta(X|ijn)$ of $Q^0_n(M)$. A fragment of the operation is illustrated in the
picture.

  \vspace{0cm}
\begin{center}
\includegraphics[scale=1]{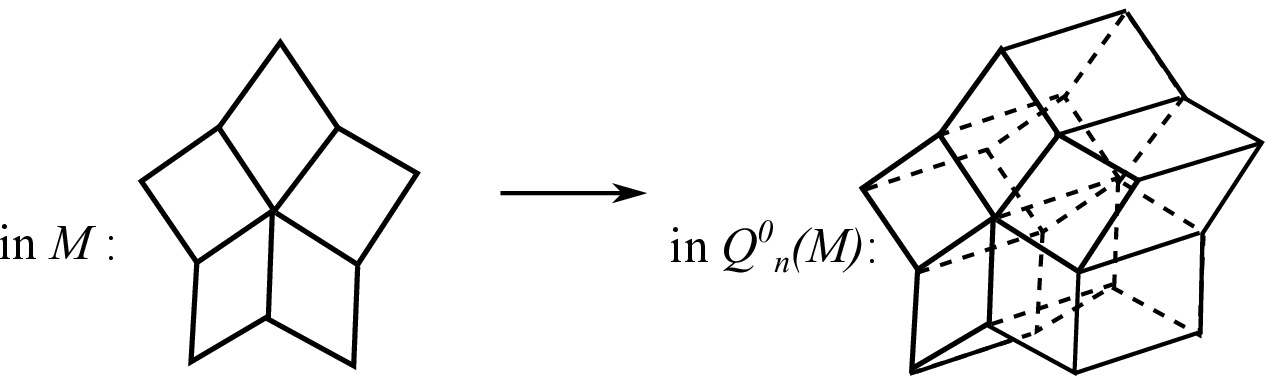}
\end{center}
\vspace{0cm}

Using the facts that $\Mfr=M$ and that the boundary cycle of $M$ is the rim of
$Z'$, one can see that
  \begin{numitem1} \label{eq:Q-Q+}
taken together, the sets of cubes in $Q^-(M)$, $Q^0_n(M)$ and
$\{\zeta+\theta_n\colon \zeta\in Q^+(M)\}$ form a cubillage in $Z=Z(n,3)$.
  \end{numitem1}

We denote this cubillage as $Q(M)=Q_n(Q',M)$ and call it the
$n$-\emph{expansion} of $Q'$ using $M$. There is a natural relation between the
$n$-contraction and $n$-expansion operations, as follows. (A proof is
straightforward and left to the reader as an exercise.)
  \begin{prop} \label{pr:contr-exp}
The correspondence $(Q',M)\mapsto Q(M)$, where $Q'$ is a cubillage on
$Z(n-1,3)$, $M$ is an s-membrane in $Q'$, and $Q(M)$ is the $n$-expansion of
$Q'$ using $M$, gives a bijection between the set of such pairs $(Q',M)$ in
$Z(n-1,3)$ and the set of cubillages in $Z(n,3)$. Under this correspondence,
$Q'$ is the $n$-contraction of $Q=Q(M)$ and $M$ is the image of the $n$-pie in
$Q$ under the $n$-contraction operation. \qed
  \end{prop}

We will also take advantages from handling the 1-pie of a cubillage $Q$ on
$Z(n,3)$ and applying the corresponding \emph{1-contraction} and
\emph{1-expansion} operations, which are symmetric to those concerning the
color $n$ as above. More precisely, if we make a mirror reflection of $\Theta$
by replacing each generator $\theta_i=(x_i,y_i,1)$ by $(-x_i,y_i,1)$, denoted
as $\theta'_{n+1-i}$, then the 1-edges of $Q$ turn into $n$-edges of the
corresponding cubillage $Q'$ on $Z(\theta'_1,\ldots,\theta'_n)$, and the 1-pie
of $Q$ turns into the $n$-pie of $Q'$. This leads to the corresponding
counterparts of~\refeq{C-Cp}--\refeq{Q-Q+} and Proposition~\ref{pr:contr-exp}.
 \medskip

\noindent\textbf{Remark 2.} The usage of $j$-pies and operations on them is less
advantageous when $1<j<n$. The trouble is that if a cube $\zeta$ containing a
$j$-edge is of the form $\zeta(X|ijk)$, where $i<j<k$, then $\zeta$ is the sum
of the rhombus $\rho=\rho(X|ik)$ and segment $[0,\theta_j]$, but $\rho$ lies on
the rear side of $\zeta$. For this reason, a relation between $j$-pies and
rhombus tilings becomes less visualized. However, we will not use $j$-pies
with $1<j<n$ in this paper.

\subsection{Applications of the contraction and expansion operations}  \label{ssec:appl_contr}

We start with the following assertion, using terminology and notation as above.

\begin{prop} \label{pr:edge_rh_cube}
Let $Q$ be a cubillage on $Z=Z(n,3)$.
 \begin{itemize}
\item[\rm(i)] If $Q$ contains vertices $X$ and $Xi$, then it does the edge
$(X,Xi)$.
\item[\rm(ii)] If $Q$ contains vertices $X,Xi,Xj,Xij$ ($i<j$), then it does the
rhombus $\rho(X|ij)$.
\item[\rm(iii)] If $Q$ contains a set $S$ of eight vertices
$X,Xi,Xj,Xk,Xij,Xik,Xjk,Xijk$ ($i<j<k$), then it does the cube $\zeta(X|ijk)$.
  \end{itemize}
  \end{prop}
 \begin{proof}
We use induction on $n$. Let us prove~(iii), denoting by $Q'$ the cubillage on
$Z(n-1,3)$ that is the $n$-contraction of $Q$. Three cases are possible.
 \smallskip

(a) Suppose that $k<n$ and $n\notin X$. Then $S$ belongs to the vertex set of
the subcubillage $Q^-_n$ (cf.~\refeq{Z-Z+}) and, therefore, to the vertex set
of $Q'$. By induction, $Q'$ contains the cube on $S$, namely,
$\zeta=\zeta(X|ijk)$. From~\refeq{Q-Q+} and Proposition~\ref{pr:contr-exp} it
follows that under the $n$-expansion operation for $Q'$ using the s-membrane
$M:=\Pi^-_n$, $\zeta$ becomes a cube in $Q$, as required. (Recall that
$\Pi^-_n$ is the corresponding disk in $\Pifr_n$, defined in the paragraph
before~\refeq{Z-Z+}.)

(b) Suppose that $n\in X$. Then $k<n$ and $S$ belongs to the vertex set of
$Q^+_n$. Therefore, $S':=\{Y-n\colon Y\in S\}$ is contained in $V_{Q'}$ and,
moreover, in the vertex set of the subcubillage $Q^+(M)$ of $Q'$ (where $M$ is
as in~(a)). So, by induction, $Q^+(M)$ contains the cube $\zeta'=\zeta(X-n|ijk)$.
The $n$-expansion operation for $Q'$ using $M$ transfers $\zeta'$ to the
desired cube $\zeta(X|ijk)$ in $Q$.

(c) Now let $n\notin X$ and $k=n$. Then the set $S^-:=\{X,Xi,Xj,Xij\}$ belongs
to $\Pi^-_n$ (and $Q^-_n$), and the set $S^+:=\{Xn,Xin,Xjn,Xijn\}$ to $\Pi^+_n$
(and $Q^+_n$). The $n$-contraction operation shifts $S^+$ by $-\theta_n$ and
merges it with $S^-$ (which lies in $M$). By induction, $Q'$ contains the rhombus
$\rho=\rho(X|ij)$. The $n$-expansion operation for $Q'$ using $M$ transforms
$\rho$ into the cube $\zeta(Z|ijk)$ in $Q^0_n(M)$, and therefore, in $Q$
(cf.~\refeq{Q-Q+}), as required.

Assertions in (i) and (ii) are shown in a similar fashion (even easier).
 \end{proof}

Based on this proposition, we now prove Proposition~\ref{pr:T-smembr}.

Let $Q$ be a cubillage on $Z(n,3)$, and $T$ a rhombus tiling on $Z(n,2)$ with
$V_T\subset V_Q$ (regarding vertices as subsets of $[n]$). For each rhombus
$\rho=\rho(X|ij)$ in $T$, the vertices of the form $X,Xi,Xj,Xij$ are contained
in $Q$ as well, and by~(ii) in Proposition~\ref{pr:edge_rh_cube}, $Q$ contains
a rhombus $\rho'$ on these vertices. Then $\rho=\pi(\rho')$. Combining such
rhombi $\rho'$ in $Q$ determined by the rhombi $\rho$ on $T$, we obtain a
2-dimensional subcomplex $M$ in $Q$ which is bijectively mapped by $\pi$ onto
$T$. Hence $M$ is an s-membrane in $Q$ isomorphic to $T$, yielding
Proposition~\ref{pr:T-smembr} and Theorem~\ref{tm:smembr-scoll}.


\section{The lattice of s-membranes} \label{sec:lattice_s}

As mentioned in the Introduction, the set $\Sbold^\ast(C)$ of maximal by size
strongly separated collections $S\subset 2^{[n]}$ that are contained in a fixed
maximal chord separated collection $C\subset 2^{[n]}$ has nice structural
properties. Due to
Theorems~\ref{tm:LZ},\,\ref{tm:galash},\,\ref{tm:smembr-scoll}, it is more
enlightening to deal with equivalent geometric objects, by considering a
cubillage $Q$ on the zonotope $Z=Z(n,3)$ and the set $\Mscr(Q)$ of s-membranes
in $Q$.

Using notation as in Sect.~\SSEC{pies}, for an s-membrane $M\in \Mscr(Q)$, we
write $Z^-(M)$ ($Z^+(M)$) for the (closed) region of $Z$ bounded by the front
side $\Zfr$ of $Z$ and $M$ (resp. by $M$ and the rear side $\Zrear$)) and write
$Q^-(M)$ ($Q^+(M)$) for the set of cubes of $Q$ contained in $Z^-(M)$ (resp.
$Z^+(M)$). The sets $Q^-(M)$ and $Q^+(M)$ are important in our analysis and we
call them the \emph{front heap} and the \emph{rear heap} of $M$, respectively.


Consider two s-membranes  $M,M'\in\Mscr(Q)$ and form the sets $N:=(M\cup M')^{\rm
fr}$ and $N':=(M\cup M')^{\rm re}$. Then both $N,N'$ are bijective to $Z'$ by $\pi$. Also one
can see that for any rhombus $\rho$ in $M$, if some interior point of $\rho$
belongs to $N$ ($N'$), then the entire $\rho$ lies in $N$ (resp. $N'$), and
similarly for $M'$. These observations imply that:
  \begin{numitem1}  \label{eq:NNp}
  \begin{itemize}
\item[(i)] both $N$ and $N'$ are s-membranes in $Q$;
\item[(ii)] the front heap $Q^-(N)$ of $N$ is equal to $Q^-(M)\cap Q^-(M')$,
and the front heap $Q^-(N')$ of $N'$ is equal to $Q^-(M)\cup Q^-(M')$.
  \end{itemize}
  \end{numitem1}
(Accordingly, the rear heaps of $N$ and $N'$ are $Q^+(N)=Q^+(M)\cup Q^+(M')$
and $Q^+(N')=Q^+(M)\cap Q^+(M')$.) By~\refeq{NNp}, the front heaps of
s-membranes constitute a distributive lattice, which gives rise to a similar
property for the s-membranes themselves.

  \begin{prop} \label{pr:latticeMQ}
The set $\Mscr(Q)$ of s-membranes in $Q$ is endowed with the structure of
distributive lattice in which the meet and join operations for
$M,M'\in\Mscr(Q)$ produce the s-membranes $M\wedge M'$ and $M\vee M'$ such that
$Q^-(M\wedge M')=Q^-(M)\cap Q^-(M')$ and $Q^-(M\vee M')=Q^-(M)\cup Q^-(M')$.
\hfill\qed
  \end{prop}

It is useful to give an alternative description for this lattice, which reveals
an intrinsic structure and a connection with flips in rhombus tilings. It is
based on a natural partial order on $Q$ defined below. Recall that
for a cube $\zeta$, the front side $\zetafr$ and the rear side $\zetarear$ are
formed by the rhombi as indicated in the Example in Sect.~\SSEC{smembr}.
 \medskip

\noindent\textbf{Definition.} For $\zeta,\zeta'\in Q$, we say that $\zeta$
\emph{immediately precedes} $\zeta'$ if $\zetarear\cap(\zeta')^{\rm fr}$
contains a rhombus.

  \begin{lemma} \label{lm:acyclic}
The directed graph $\Gamma_Q$ whose vertices are the cubes of $Q$ and whose
edges are the pairs $(\zeta,\zeta')$ such that $\zeta$ immediately precedes
$\zeta'$ is acyclic.
  \end{lemma}
  \begin{proof}
Consider a directed path $P=(\zeta_0,e_1,\zeta_1,\ldots,e_p,\zeta_p)$ in
$\Gamma_Q$ (where $e_r$ is the edge $(\zeta_{r-1},\zeta_r)$). We show that $P$
is not a cycle (i.e., $\zeta_0\ne\zeta_p$ when $p>0$) by using induction on
$n$. This is trivial if $n=3$.

We know that for any $n$-cube $\zeta=\zeta(X|ijn)$ of $Q$, its front rhombus
$\rho(X|ij)$ belongs to the front side $\Pifr_n$ of the $n$-pie $\Pi_n$ of $Q$,
its rear rhombus $\rho(Xn|ij)$ belongs to the rear side $\Pirear_n$, and the
other rhombi of $\zeta$, namely, $\rho(X|in),\,\rho(X|jn), \, \rho(Xj|in), \,
\rho(Xi|jn)$ lie in the interior or belt of $\Pi_n$ (for definitions, see
Sect.~\SSEC{pies}). This implies that if for some $r$, the cubes
$\zeta_{r-1}$ and $\zeta_r$ belong to different sets among $Q^-_n,\, Q^+_n,\,
\Pi_n$, then the edge $e_r$ goes either from $Q^-_n$ to $\Pi_n$, or from
$\Pi_n$ to $Q^+_n$. Therefore, $P$ crosses each of the disks  $\Pifr_n$ and
$\Pirear_n$ at most once, implying that $\zeta_0=\zeta_p$ would be possible
only if the vertices of $P$ are entirely contained in exactly one of $Q^-_n, \,
Q^+_n,\, \Pi_n$.

Let $Q'$ be the $n$-contraction of $Q$. We assume by induction that
$\Gamma_{Q'}$ is acyclic. Then the cases $\zeta_r\in Q^-_n$ and $\zeta_r\in
Q^+_n$ are impossible (subject to $\zeta_0=\zeta_p$), taking into account that
$Q'$ is obtained by combining $Q^-_n$ and the cubes of $Q^+_n$ shifted by
$-\theta_n$.

It remains to show that the subgraph $\Gamma'$ of $\Gamma_Q$ induced by the
cubes of $\Pi_n$ is acyclic. To see this, observe that for an $n$-cube
$\zeta=\zeta(X|ijn)$ of $Q$, its rear rhombi lying in the interior or belt of
$\Pi_n$ are $\rho_1:=\rho(X|in)$ and $\rho_2:=\rho(Xi|jn)$. So if $\zeta$ and
another $n$-rhombus $\zeta'$ are connected by edge $(\zeta,\zeta')$ in $\Gamma'$, then
$\zeta'$ shares with $\zeta$ either $\rho_1$ or $\rho_2$. Let us associate with
$\zeta,\zeta'$ the corresponding rhombi $\rho,\rho'$ on $\Pifr_n$,
respectively. Then $\rho=\rho(X|ij)$ and $\rho'=(X'|i'j')$ for some $X'$ and
$i'<j'$. These rhombi have an edge in common; namely, the edge $e=(X,Xi)$ if
$\zeta,\zeta'$ share the rhombus $\rho_1$, and the edge $e'=(Xi,Xij)$ if
$\zeta,\zeta'$ share $\rho_2$. Note that (under the projection by $\pi$) both
$e,e'$ belong to the \emph{left} boundary of $\rho$, in view of
$i<j$.

These observations show that the subgraph $\Gamma'$ is isomorphic to the graph
$\Gamma''$ whose vertices are the rhombi in $\Pifr_n$ and whose edges are the
pairs $(\rho,\rho')$ such that $\rho,\rho'$ share an edge lying in the left
boundary of $\rho$ (and in the right boundary of $\rho'$). This $\Gamma''$ is
acyclic. (Indeed, consider the rhombus tiling $T:=\pi(\Pifr_n)$ on $Z(n-1,2)$.
Then any directed path from $\emptyset$ to $[n-1]$ in $T$ may be crossed by
an edge of $\Gamma''$ only from right to left, not back,
whence $\Gamma''$ is acyclic.)
  \end{proof}

 \begin{corollary} \label{cor:part_order}
The graph $\Gamma_Q$ induces a partial order $\prec$ on the cubes of $Q$.
Moreover, the ideals of $(Q,\prec)$ (i.e., the subsets $Q'\subseteq Q$
satisfying $(\zeta\in Q',\; \zeta'\prec \zeta \Longrightarrow \zeta'\in Q')$
are exactly the front heaps $Q^-(M)$ of s-membranes $M\in\Mscr(Q)$. \hfill\qed
  \end{corollary}

Here the second assertion can be concluded from the fact that the ideals of
$(Q,\prec)$ are the sets of cubes $Q'\subseteq Q$ such that $\Gamma_Q$ has no
edge going from $Q-Q'$ to $Q'$.

Using Corollary~\ref{cor:part_order}, we now explain a  relation to strong
flips in rhombus tilings. For convenience we identify an s-membrane
$M\in\Mscr(Q)$ with the corresponding rhombus tiling $\pi(M)$ on $Z(n,2)$. In
particular, the minimal s-membrane $\Zfr$ is identified with the standard
tiling $\Tst_n$, and the maximal s-membrane $\Zrear$ with the
antistandard tiling $\Tant_n$.

Let $M\in\Mscr(Q)$ be different from $\Tst_n$. Then the heap $J:=Q^-(M)$
is nonempty. Since $\Gamma_Q$ is acyclic, $J$ has a maximal element
$\zeta=\zeta(X|ijk)$  (i.e., there is no $\zeta'\in J$ with $\zeta\prec
\zeta'$). Then $M$ contains all rear rhombi of $\zeta$, namely, $\rho(X|ik),\;
\rho(Xi|jk),\; \rho(Xk|ij)$. They span the hexagon $H(X|ijk)$ having
$\Yturn$-configuration and we observe that
  \begin{numitem1} \label{eq:flip_in_M}
for $M,\,J,\, \zeta$ as above, the set $J':=J-\{\zeta\}$
is an ideal of $(Q,\prec)$ as well, and the s-membrane (rhombus tiling) $M'$
with $Q^-(M')=J'$ is obtained from $M$ by replacing the rhombi of
$\zetarear$ by the rhombi forming $\zetafr$ (namely, $\rho(X|ij),\;
\rho(X|jk),\;\rho(Xj|ik)$), or, in other words, by the lowering flip involving the
hexagon $H(X|ijk)$ (see the picture in the end of Sect.~\SSEC{rhomb_til}).
  \end{numitem1}

(Of especial interest are \emph{principal} ideals of $(Q,\prec)$; each of them
is determined by a cube $\zeta\in Q$ and consists of all $\zeta'\in Q$ from
which $\zeta$ is reachable by a directed path in $\Gamma_Q$. The s-membrane
corresponding to such an ideal admits only one lowering flip within $Q$,
namely, that determined by the rhombi of $\zeta$. Symmetrically: considering
$M\in\Mscr(Q)$ different from $\Tant_n$ and its rear heap $R:=Q^+(M)$, and
choosing an element $\zeta\in R$ that admits no $\zeta'\in R$ with $\zeta'\prec
\zeta$, we can make the raising flip by replacing in $M$ the rhombi of
$\zetafr$ by the ones of $\zetarear$. When $R$ is formed by some $\zeta\in Q$
and all $\zeta'\in Q$ reachable from $\zeta$ by a directed path in $\Gamma_Q$,
then $M$ admits only one raising flip within $Q$, namely, that
determined by the rhombi of $\zeta$.)

In terms of maximal s-collections, \refeq{flip_in_M} together with
Proposition~\ref{pr:latticeMQ} implies the following
  \begin{corollary} \label{cor:flip_in_SC}
Let $C$ be a maximal chord separated collection in $2^{[n]}$. The set
$\Sbold^\ast(C)$ of maximal by size strongly separated collections in $C$ is a
distributive lattice with the minimal element $\Iscr_n$ and the maximal element
\rm{co-}$\Iscr_n$ (being the set of intervals and the set of co-intervals in
$[n]$, respectively) in which $S\in\Sbold^\ast(C)$ immediately precedes
$S'\in\Sbold^\ast(C)$ if and only if $S'$ is obtained from $S$ by one raising
flip (``in the presence of six witnesses'').
  \end{corollary}

\noindent\textbf{Remark 3.} Note that the set of all maximal s-collections in
$2^{[n]}$ is a poset (with the unique minimal element $\Iscr_n$ and the unique
maximal element \rm{co-}$\Iscr_n$ and with the immediate preceding relation
given by strong flips as well); however, in contrast to $\Sbold^\ast(C)$, this
poset is not a lattice already for $n=6$, as is shown in Ziegler~\cite{zieg}.

Note also that a triple $\tau$ of rhombi in an s-membrane $M\in\Mscr(Q)$ that
spans a hexagon need not belong to one cube of $Q$ (in contrast to~(iii) in
Proposition~\ref{pr:edge_rh_cube} where $Q$ contains a cube if all \emph{eight}
vertices of this cube belong to $V_Q$). In this case, $(M,\tau)$ determines a
flip in the variety of all rhombus tilings on $Z(n,2)$ but not within
$\Mscr(Q)$.


\section{Embedding rhombus tilings in cubillages} \label{sec:embed_rt}

In this section we study cubillages on $Z(n,3)$ containing one or more fixed
s-membranes.

 \subsection{Extending an s-membrane to a cubillage}  \label{ssec:memb-to-cube}
We start with the following issue. Given a maximal strongly separated
collection $S\subset 2^{[n]}$, let $\Cbold(S)$ be the set of maximal chord
separated collections containing $S$. How to construct explicitly one instance
of such c-collections?  A naive method consists in choosing, step by step, a
new subset $X$ of $[n]$ and adding it to the current collection including $S$
whenever $X$ is chord separated from all its members. However, this method is
expensive as it may take exponentially many (w.r.t. $n$) steps.

An efficient approach involving geometric interpretations and using flips in
s-membranes consists in the following. We build in the ``empty'' zonotope
$Z=Z(n,3)$ the abstract s-membrane $M=M(S)$ with $V_M=S$, by embedding $S$ (as
the corresponding set of points) in $Z$ and forming the rhombus $\rho(X|ij)$
for each quadruple of the form $\{X,Xi,Xj,Xij\}$ in $S$. Then we construct a
cubillage $Q$ containing this s-membrane (thus obtaining $S\subset
V_Q\in\Cbold(S)$, as required).

This is performed in two stages. At the first stage, assuming that $M$ is
different from $\Zfr$ (equivalently, $\pi(M)\ne \Tst_n$), we grow, step by
step, a partial cubillage $Q'$ filling the region $Z^-(M)$ between $\Zfr$ and
$M$, starting with $Q':=\emptyset$. At each step, the current $Q'$ is such that
$(Q')^{\rm re}=M$ and $(Q')^{\rm fr}$ forms an s-membrane $M'$. If $M'=\Zfr$,
we are done. Otherwise $\pi(M')\ne\Tst$ implies that $M'$ contains at least one
triple of rhombi spanning a hexagon having $\Yturn$-configuration (see the end
of Sect.~\SSEC{rhomb_til}). We choose one hexagon $H=H(X|ijk)$ of this sort,
add to $Q'$ the cube $\zeta=\zeta(X|ijk)$ induced by $H$, and update $M'$
accordingly by replacing the rhombi of $H$ by the other three rhombi in $\zeta$
(which form $\zetafr$); we say that the updated $M'$ is obtained from the
previous one by the \emph{lowering flip using} $\zeta$. And so on until we
reach $\Zfr$.

At the second stage, acting in a similar way, we construct a partial cubillage
$Q''$ filling the region $Z^+(M)$ between $M$ and $\Zrear$. Namely, a current
$Q''$ is such that $(Q'')^{\rm fr}=M$, and $(Q'')^{\rm re}$ forms an s-membrane
$M''$. Unless $M''=\Zrear$, we choose in $M''$ a hexagon $H$ having
$Y$-configuration, add to $Q''$ the cube $\zeta$ induced by $H$ and update
$M''$ accordingly, thus making the \emph{raising flip using} $\zeta$. And so
on.

Eventually, $Q:=Q'\cup Q''$ becomes a complete cubillage in $Z$ containing $M$,
as required. Since the partial cubillages $Q',Q''$ are constructed
independently,
  \begin{numitem1} \label{eq:direct_union}
the set $\Qbold(M)$ of cubillages on $Z=Z(n,3)$ containing a fixed s-membrane
$M$ is represented as the ``direct union'' of the sets $\Qbold^-(M)$ and
$\Qbold^+(M)$ of partial cubillages filling $Z^-(M)$ and $Z^+(M)$,
respectively, i.e., $\Qbold(M)=\{Q'\cup Q''\colon Q'\in\Qbold^-(M),\; Q''\in
\Qbold^+(M)\}$.
  \end{numitem1}

\noindent\textbf{Remark 4.} When $M=\Zfr$ ($M=\Zrear$), ~$\Qbold^+(M)$ (resp.
$\Qbold^-(M)$) becomes the set $\Qbold_n$ of all cubillages on $Z(n,3)$. The
latter set is connected by 3-flips (defined in the Introduction); moreover, a
similar property is valid for fine tilings on zonotopes of any dimension, as a
consequence of results in~\cite{MS}.

It light of this, one can consider an arbitrary s-membrane $M$ and ask about the
connectedness of the set $\Qbold(M)$ of cubillages w.r.t. 3-flips that preserve $M$.
This is equivalent to asking whether or not for any two partial cubillages
$Q,Q'\in\Qbold^-(M)$, there exists a sequence $Q_0,Q_1,\ldots,Q_p\in
\Qbold^-(M)$ such that $Q_0=Q$, $Q_p=Q'$, and each $Q_r$ ($r=1,\ldots,p$) is
obtained by a 3-flip from $Q_{r-1}$; and similarly for $\Qbold^+(M)$. The answer
to this question is affirmative (a proof is left to a forthcoming paper).

 \subsection{Cubillages for two s-membranes}  \label{ssec:two_smembr}

One can address the following issue. Suppose we are given two abstract
s-membranes $M,M'$ properly embedded in $Z=Z(n,3)$. When there exists a
cubillage $Q$ on $Z$ containing both $M$ and $M'$? The answer is clear: if and
only if the set $V_M\cup V_{M'}$ is chord separated. However, one can ask:  how
to construct such a $Q$ efficiently?

For simplicity, consider the case of ``non-crossing'' s-membranes, assuming
that $M$ is situated in $Z$ before $M'$, i.e., $M$ lies in $Z^-(M')$.

A partial cubillage $Q'$ filling $Z^-(M)$ and a partial cubillage $Q''$ filling
$Z^+(M')$ always exist and can be constructed by the method as in
Sect.~\SSEC{memb-to-cube}. So the only problem is to construct a partial
cubillage $\tilde Q$ filling the space between $M$ and $M'$, i.e.,
$Z(M,M'):=Z^+(M)\cap Z^-(M')$; then $Q:=Q'\cup\tilde Q\cup Q''$ is as required.
Conditions when a required $\tilde Q$ does exist are exposed in the proposition
below.

We need some definitions. Consider a rhombus tiling $T$ on the zonogon
$Z'=Z(n,2)$ and a color $i\in[n]$. For each $i$-edge $e$ in $T$, let $m(e)$ be
the middle point on $e$, and for each $j\in [n]-\{i\}$, let $c(\rho$) be the
central point of the $\{i,j\}$-rhombus $\rho$ in $T$ (where $\rho$ is the
$ij$-rhombus when $i<j$, and the $ji$-rhombus when $j<i$). For such a $\rho$
and the $i$-edges in it, say, $e$ and $e'$, connect $c(\rho)$ by straight-line
segments with each of $m(e)$ and $m(e')$. One easily shows that the union of
these segments over all $j$ produces a non-self-intersecting piecewise linear
curve connecting the middle points of the two $i$-edges on the left and right
boundaries of $Z'$, denoted as $D_i$ and called $i$-th (undirected) \emph{dual
path} for $T$. (The set $\{D_1,\ldots, D_n\}$ matches a pseudo-line
arrangement, in a sense.)
 \medskip

\noindent\textbf{Definitions.} Let $1\le i<j<k\le n$ and let $\rho$ be the
$ik$-rhombus in $T$. The triple $ijk$ is called \emph{normal} if $\rho$ lies
above $D_j$, and an \emph{inversion} for $T$ if $\rho$ lies below $D_j$. The
set of inversions for $T$ is denoted by $\Inver(T)$. Also we say that a triple
$ijk$ in $T$ is \emph{elementary} if the rhombi of types $ij$, $ik$ and $jk$ in it
span a hexagon (which has $Y$-configuration if $ijk$ is normal, and
$\Yturn$-configuration if $ijk$ is an inversion).
  \smallskip

See the picture where a normal triple (an inversion) $ijk$ is illustrated in
the left (resp. right) fragment and the corresponding dual paths are drawn by
dotted lines.

  \vspace{0cm}
\begin{center}
\includegraphics[scale=1]{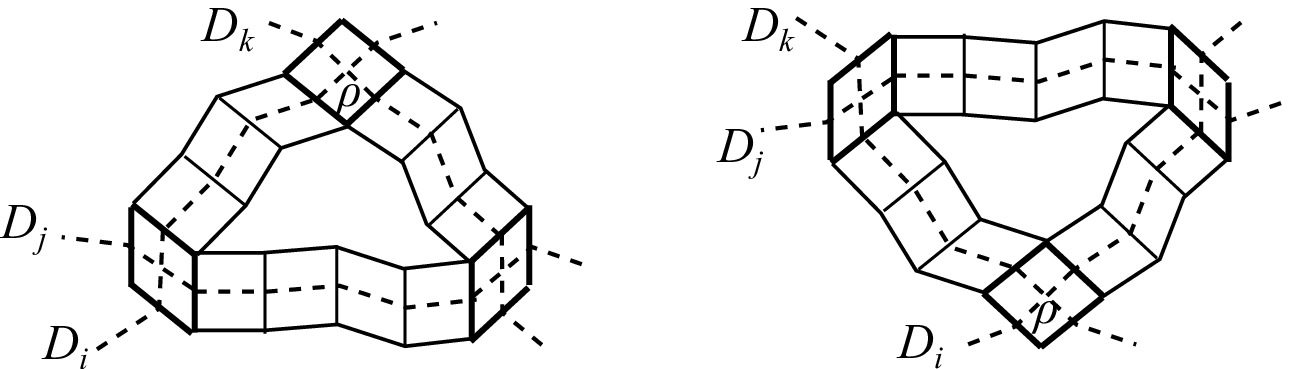}
\end{center}
\vspace{-0.3cm}

  \begin{prop} \label{pr:two_smbr}
Let $M,M'$ be two s-membranes in $Z=Z(n,3)$ such that $M\subset Z^-(M')$. Then
a partial cubillage $\tilde Q$ filling $Z(M,M')$ (and therefore a cubillage
on $Z$ containing both $M,M'$) exists if and only if
$\Inver(M)\subseteq\Inver(M')$. Such a $\tilde Q$ consists of
$|\Inver(M')|-|\Inver(M)|$ cubes and can be constructed efficiently.
  \end{prop}

One direction in this proposition is rather easy. Indeed, suppose that a
partial cubillage $\tilde Q$ filling $Z(M,M')$ does exist. Take a minimal
(w.r.t. the order $\prec$ as in Sect.~\SEC{lattice_s}) cube
$\zeta=\zeta(X|ijk)$ in $\tilde Q$. Then the rhombi of $\zetafr$ belong to $M$
and span the hexagon $H=H(X|ijk)$ having $Y$-configuration. Hence the triple
$ijk$  in $M$ is normal and elementary (using terminology for $M$ as that for
the tiling $\pi(M)$). By making the flip in $M$ using $\zeta$, we obtain an
s-membrane in which $ijk$ becomes an inversion, and the fact that $ijk$ is
elementary implies that no other triple $i'j'k'$ changes its status. Also the
new s-membrane becomes closer to $M'$. Applying the procedure $|\tilde Q|$
times, we reach $M'$. This shows ``only if'' part in
Proposition~\ref{pr:two_smbr}.

As to ``if'' part, its proof is less trivial and relies on a result by Felsner and
Weil. Answering an open question by Ziegler~\cite{zieg}, they proved the
following assertion (stated in~\cite{FW} in equivalent terms of pseudo-line
arrangements).
  \begin{theorem} {\rm\cite{FW}} \label{tm:FW}
Let $T,T'$ be rhombus tilings on $Z(n,2)$ and let $\Inver(T)\subset\Inver(T')$.
Then $T$ has an elementary triple contained in $\Inver(T')-\Inver(T)$.
  \end{theorem}

\noindent(This is a 2-dimensional analog of the well-known fact that for two
permutations $\sigma,\sigma'\in S_n$ with $\Inver(\sigma)\subset
\Inver(\sigma')$, $\sigma$ has a transposition in
$\Inver(\sigma')-\Inver(\sigma)$. Ziegler~\cite{zieg} showed that the
corresponding assertion in dimension 3 or more is false.)

Now Theorem~\ref{tm:FW} implies that if $M,M'$ are s-membranes with
$\Inver(M)\subset\Inver(M')$, then there exists a cube $\zeta=\zeta(X|ijk)$
such that $\zetafr\subset M$ and $ijk\in\Inver(M')-\Inver(M)$. The flip
in $M$ using $\zeta$ increases the set of inversions by $ijk$.
This enables us to recursively construct a partial cubillage filling $Z(M,M')$
starting with $\zeta$, and ``if'' part of Proposition~\ref{pr:two_smbr} follows.


\section{W-membranes and quasi-combies} \label{sec:w-membr}

In this section we deal with a maximal c-collection $C$ in $2^{[n]}$ and its
associated cubillage $Q$ on the zonotope $Z=Z(n,3)$ (i.e., with $V_Q=C$), and
consider the class $\Wbold^\ast(C)$ of maximal by size weakly separated
collections contained in $C$. (Recall that $C$ need not be w-pure, by
Lemma~\ref{lm:example}.) Since each $W\in\Wbold^\ast(C)$ is the spectrum of a
combi on the zonogon $Z'=Z(n,2)$ (cf.~Theorem~\ref{tm:combi}), a reasonable
question is how a combi $K$ with $V_K\subset V_Q$ (regarding vertices as
subsets of $[n]$) relates to the structure of $Q$. We have seen that maximal by
size s-collections in $C$ and their associated rhombus tilings on $Z'$ are
represented by s-membranes, that are special 2-dimensional subcomplexes in $Q$.
In case of weak separation, we will represent combies via \emph{w-membranes},
that are subcomplexes of a certain subdivision, or fragmentation, of $Q$. Also,
along with a combi $K$ with $V_K\subset V_Q$, we will be forced to deal with
the set of so-called \emph{quasi-combies} accompanying $K$, which were
introduced in~\cite{DKK3} and have a nice geometric interpretation in terms of
$Q$ as well.

 \subsection{Fragmentation of a cubillage and quasi-combies}  \label{ssec:fragment}

The \emph{fragmentation} $\Qfrag$ of a cubillage $Q$ on $Z=Z(n,3)$ is the
complex obtained by cutting $Q$ by the horizontal planes through the vertices
of $Q$, i.e., the planes $z=h$ for $h=0,\ldots,n$. This subdivides each cube
$\zeta=\zeta(X|ijk)$ into three pieces: the lower tetrahedron $\zeta^\nabla$,
the middle octahedron $\zeta^\square$, and the upper tetrahedron
$\zeta^\Delta$, called the $\nabla$-, $\square$-, and $\Delta$-\emph{fragments}
of $\zeta$, respectively. Depending on the context, we also may think of
$\Qfrag$ as the set of such fragments over all cubes. We say that a fragment
has \emph{height} $h+\frac12$ if it lies between the planes $z=h$ and $z=h+1$.

It is convenient to visualize faces of $\Qfrag$ as though looking at them from
the front and slightly from below, i.e., along a vector $(0,1,\eps)$, and
accordingly use the projection $\pi_\eps: \Rset^3\to\Rset^2$ defined by
$\pi_\eps(x,y,z)=(x,z-\eps y)$ for a sufficiently small $\eps>0$. One can see
that $\pi_\eps$ transforms the generators $\theta_1, \ldots,\theta_n$ for $Z$
as in~\refeq{cyc_conf} into generators for $Z'=Z(n,2)$) which are adapted for
combies, i.e., satisfy the strict convexity condition~\refeq{strict_xi}.

For $S\subset Z$, let $S^{\rm fr}_\eps$ ($S^{\rm re}_\eps$) denote the set of
points of $S$ seen from the front (from the rear) in the direction related to
$\pi_\eps$, i.e., the points $(x,y,z)\in S\cap\pi_\eps^{-1}(\alpha,\beta)$ with
$y$ minimum (resp. maximum), for all $(\alpha,\beta)\in\Rset^2$. In particular,
when replacing the previous projection $\pi$ by $\pi_\eps$, all facets
(triangles) of the fragments of a cube become fully seen from the front or
rear; see the picture.

  \vspace{0cm}
\begin{center}
\includegraphics[scale=1]{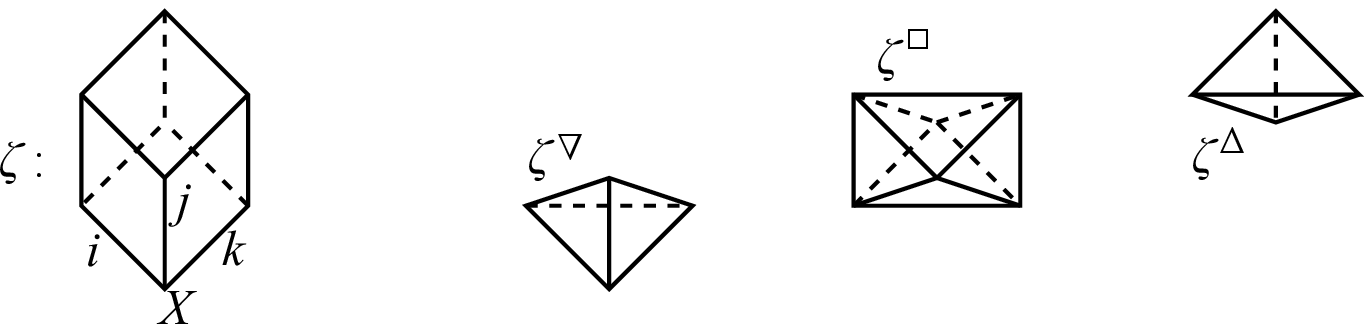}
\end{center}
\vspace{-0.3cm}

Thus, all 2-dimensional faces in $\Qfrag$ are triangles, and we conditionally
refer to those of them that lie in horizontal sections $z=h$ as \emph{horizontal}
triangles, and to the other ones (halves of rhombi in $Q$) as \emph{vertical}
ones. Horizontal triangles $\tau$ are divided into two groups. Namely, $\tau$
is called \emph{upper} (\emph{lower}) if it has vertices of the form $Xi,Xj,Xk$
(resp. $Y-k,Y-j,Y-i$) for $i<j<k$, and therefore its ``obtuse'' vertex $Xj$
(resp. $Y-j$) is situated above the edge $(Xi,Xk)$ (resp. below the edge
$(Y-k,Y-i)$), called the \emph{longest} edge of $\tau$ (which is justified when
$\eps$ is small). Equivalently, an upper (lower) horizontal $\tau$ belongs to
an $\nabla$-fragment (resp. $\Delta$-fragment).

Accordingly, we refer to the edges in horizontal sections as horizontal ones,
or \emph{H-edges}, and to the other edges as vertical ones, or \emph{V-edges}
(adapting terminology for combies from Sect.~\SSEC{combi}).

For $h\in[n]$, let $\Qfrag_h$ denote the section of $Q$ at height $h$
(consisting of horizontal triangles). The triangulation $\Qfrag_h$ partitioned
into upper and lower triangles will be of use in what follows. (A nice property
of $\Qfrag_h$ pointed out in~\cite{gal} is that its spectrum (the set of
vertices regarded as subsets of $[n]$) constitutes a maximal w-collection in
$\binom{[n]}{h}$.) For example, if $Q$ is the cubillage on $Z(4,3)$ formed by
four cubes $\zeta(\emptyset|123),\; \zeta(\emptyset|134),\; \zeta(1|234),\;
\zeta(3|124)$, then the triangulations $\Qfrag_1,\, \Qfrag_2,\,\Qfrag_3$ are as
illustrated in the picture, where the sections of these cubes are labeled by
$a,b,c,d$, respectively.

  \vspace{-0.3cm}
\begin{center}
\includegraphics[scale=0.9]{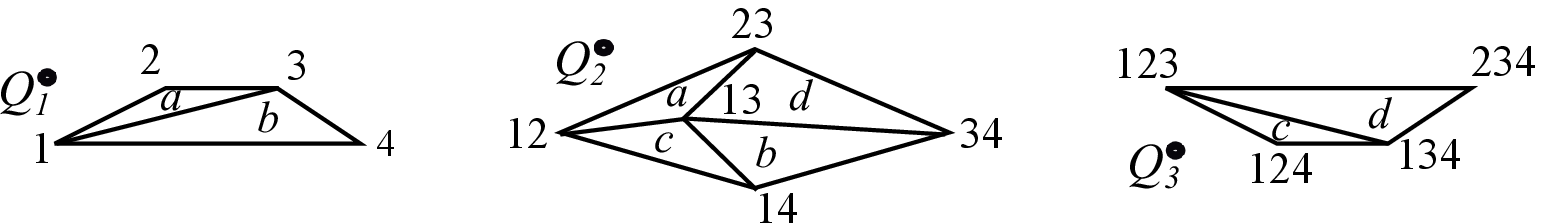}
\end{center}
\vspace{-0.3cm}

 \subsection{W-membranes}  \label{ssec:wmembran}

\noindent\textbf{Definition.} A 2-dimensional subcomplex $M$ of the
fragmentation $\Qfrag$ is called a \emph{w-membrane} if $M$ is a disk which is
bijectively projected by $\pi_\eps$ on $Z(n,2)$; equivalently, the boundary of
the disk $M$ is the rim $\Zrim$ of $Z=Z(n,3)$ and $M=\Mfr_\eps$.
 \medskip

Arguing as in Sect.~\SEC{lattice_s} for s-membranes, one shows
that the set $\Mscr(\Qfrag)$ of w-membranes in $\Qfrag$ constitutes a
distributive lattice.

More precisely, associate with a w-membrane $M$: (a) the part $Z^-(M)$
($Z^+(M)$) of $Z$ between $\Zfr$ and $M$ (resp. between $M$ and $\Zrear$); and (b)
the subcomplex $\Qfragmin(M)$ ($\Qfragpl(M)$) of $\Qfrag$ contained in
$Z^-(M)$ (resp. $Z^+(M)$), called the \emph{front heap} (resp. \emph{rear
heap}) when it is regarded as the corresponding set of $\nabla$-, $\square$-,
and $\Delta$-fragments.

Then (similar to~\refeq{NNp}) for two w-membranes $M,M'\in\Mscr(\Qfrag)$, we
have:
  \begin{numitem1} \label{eq:NNpeps}
  \begin{itemize}
  \item[(i)] both $N:=(M\cup M')^{\rm fr}_\eps$ and $N':=(M\cup M')^{\rm re}_\eps$
are w-membranes;
  \item[(ii)] $\Qfragmin_N=\Qfragmin_M\cap\Qfragmin_{M'}$ and
$\Qfragmin_{N'}=\Qfragmin_M\cup\Qfragmin_{M'}$.
  \end{itemize}
  \end{numitem1}

  \begin{prop} \label{pr:lattfragm}
$\Mscr(\Qfrag)$ is a distributive lattice in which operations $\wedge$ and
$\vee$ applied to $M,M'\in\Mscr(\Qfrag)$ produce w-membranes $M\wedge M'$ and $M\vee
M'$ such that $\Qfragmin(M\wedge M')=\Qfragmin(M)\cap \Qfragmin(M')$ and
$\Qfragmin(M\vee M')=\Qfragmin(M)\cup \Qfragmin(M')$. \hfill\qed
  \end{prop}

Next, for fragments $\tau,\tau'$ in $\Qfrag$, we say that $\tau$
\emph{immediately precedes} $\tau'$ if $\taurear_\eps \cap (\tau')^{\rm
fr}_\eps$ consists of a (vertical or horizontal) triangle. Accordingly, we
define the directed graph $\Gamma_{\Qfrag}$ whose vertices are the fragments in
$\Qfrag$ and whose edges are the pairs $(\tau,\tau')$ such that $\tau$
immediately precedes $\tau'$.
  \begin{lemma} \label{lm:acycfrag}
The graph $\Gamma_{\Qfrag}$ is acyclic.
  \end{lemma}
  \begin{proof}
Consider a directed path $P=(\tau_0,e_1,\tau_1,\ldots,e_p,\tau_p)$ in
$\Gamma_{\Qfrag}$. We show that $P$ is not a cycle as follows.

If consecutive fragments $\tau=\tau_{i-1}$ and $\tau'=\tau_i$ share a
horizontal triangle $\sigma$ of height $h$ (i.e., lying in the plane $z=h$),
then the construction of $\pi_\eps$ together with the equality
$\sigma=\taurear_\eps \cap (\tau')^{\rm fr}_\eps$ implies that $\tau$ lies
below and $\tau'$ lies above the plane $z=h$.
On the other hand, if $\tau$ and $\tau'$ share a vertical triangle, then both
$\tau,\tau'$ have the same height.

Thus, it suffices to show that if all fragments $\tau_i$ in $P$ have the same
height, then $P$ is not a cycle. This assertion follows from
Lemma~\ref{lm:acyclic} and the observation that if fragments $\tau,\tau'$ of
$\Qfrag$ share a vertical triangle $\sigma$, and $\tau$ immediately precedes
$\tau'$, then the cubes $\zeta,\zeta'$ of $Q$ containing these fragments
(respectively) share the rhombus $\rho$ including $\sigma$ and such that
$\rho=\zetarear\cap(\zeta')^{\rm fr}$.
  \end{proof}

 \begin{corollary} \label{cor:ideal_frag}
The graph $\Gamma_{\Qfrag}$ induces a partial order $\prec$ on the fragments of
$\Qfrag$. The ideals of $(\Qfrag,\prec)$ are exactly the front heaps
$\Qfragmin(M)$ of w-membranes $M\in\Mscr(\Qfrag)$.\hfill\qed
  \end{corollary}

When a w-membrane $M$ is different from the minimal w-membrane $\Zfr$, the
ideal $F:=\Qfragmin(M)$ has at least one maximal element, i.e., a fragment
$\tau\in F$ such that there is no $\tau'\in F-\{\tau\}$ with $\tau\prec\tau'$.
Equivalently, the rear side $\taurear_\eps$ is entirely contained in $M$.
The \emph{lowering flip} in $M$ using $\tau$ replaces the triangles of
$\taurear_\eps$ by the ones of $\taufr_\eps$, producing a w-membrane $M'$
closer to $\Zfr$, namely, such that $\Qfragmin(M')=F-\{\tau\}$. Note that this
flip preserves the set of vertices (i.e., $V_{M'}=V_M$) if $\tau$ is a
$\nabla$- or $\Delta$-fragment, in which case we refer to this as a \emph{tetrahedral}
(lowering) flip. See the picture.

  \vspace{-0.1cm}
\begin{center}
\includegraphics[scale=1]{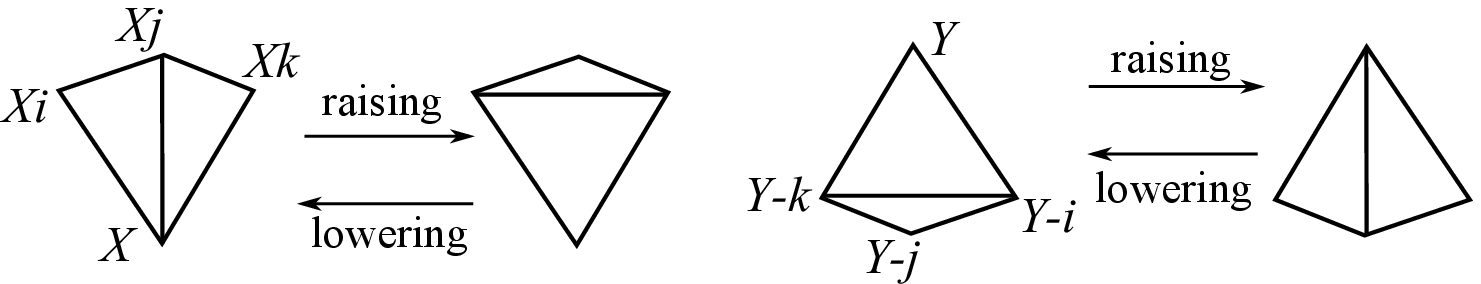}
\end{center}
\vspace{-0.1cm}

In contrast, if $\tau$ is a $\square$-fragment, then the set of vertices does
change, namely, $V_{M'}=(V_M-\{Xik\})\cup\{Xj\}$, where $\tau$ belongs to the
cube $\zeta(X|ijk)$; we refer to such a flip as \emph{octahedral} or
\emph{essential}. See the picture.

  \vspace{0cm}
\begin{center}
\includegraphics[scale=1]{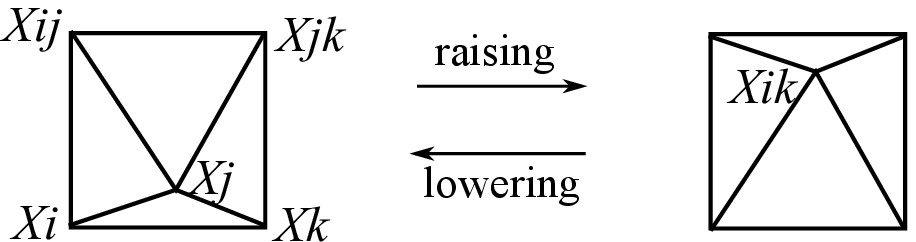}
\end{center}
\vspace{0cm}

Symmetrically, when $M\ne\Zrear$, its rear heap $R:=\Qfragpl(M)$ has at least
one minimal fragment $\tau$, i.e., such that there is no $\tau'\in R-\{\tau\}$
with $\tau'\prec \tau$. Equivalently, $\taufr_\eps$ is entirely contained in
$M$. The \emph{raising flip} in $M$ using $\tau$ produces a w-membrane $M'$
closer to $\Zrear$. Such flips, referred to as tetrahedral and octahedral (or
essential) as before, are illustrated on the above two pictures as well.

Making all possible lowering or raising \emph{tetrahedral} flips starting with
a given w-membrane $M$, we obtain a set of w-membranes with the same spectrum
$V_M$, denoted as $\Escr(M)$ and called the \emph{escort} of $M$. Of
especial interest is a w-membrane $L\in\Escr(M)$ that has the maximum number
of V-edges. Such an $L$ admits neither a $\nabla$-fragment $\tau$ with
$\taurear_\eps\subset L$, nor a $\Delta$-fragment $\tau'$ with $(\tau')^{\rm
fr}_\eps\subset L$, since a lowering flip in the former case and a raising flip
in the latter case would increase the number of V-edges. We call $L$ a
\emph{fine} w-membrane.

We shall see later that the w-membranes correspond to the so-called
\emph{non-expensive quasi-combies}, and the fine w-membranes to the
combies, which are \emph{compatible} with $\Qfrag$.
The following auxiliary statement will be of use.

  \begin{numitem1} \label{eq:vert-hor}
  \begin{itemize}
  \item[(i)] Let $\Qfrag$ contain a vertical $\Delta$-triangle $\Delta$ and a
lower horizontal triangle $\sigma$ sharing an edge $e$ that is the longest
edge of $\sigma$ (and the base edge of $\Delta$). Then $\Delta$ and $\sigma$
belong to the same $\Delta$-fragment $\tau$ of $\Qfrag$ (thus forming
$\taufr_\eps$).
  \item[(ii)] Symmetrically, if a vertical $\nabla$-triangle $\nabla$ and an
upper horizontal triangle $\sigma$ share an edge that is the longest edge of
$\sigma$, then $\nabla\cup\sigma= \taurear_\eps$ for some $\nabla$-fragment
$\tau$ of $\Qfrag$.
  \end{itemize}
  \end{numitem1}

Indeed, let $\rho$ be the rhombus in $Q$ containing the triangle $\Delta$ as
in~(i). This $\rho$ is a facet of one or two cubes of $Q$ and $\sigma$ lies in
the section of one of them, $\zeta$ say, by the horizontal plane containing
$e$. Since $\sigma$ is lower, the only possible case is when $\Delta$ and
$\sigma$ form the front side of the $\Delta$-fragment of $\zeta$, as required.
The case~(ii) is symmetric.

A useful consequence of~\refeq{vert-hor} is:
  \begin{numitem1} \label{eq:fine_wmembr}
for any horizontal triangle $\sigma$ of a fine w-membrane $L$, the longest edge
of $\sigma$ belongs to one more (lower or upper) horizontal triangle of $L$.
  \end{numitem1}

Indeed, if $\sigma$ is lower, then its longest edge belongs to neither a
vertical $\nabla$-triangle (since $\pi_\eps$ is injective on $L$), nor a
vertical $\Delta$-triangle (otherwise $\sigma\cup\Delta$ would be as
in~\refeq{vert-hor}(i) and one could make a lowering flip increasing the number
of V-edges). When $\sigma$ is upper, the argument is similar
(using~\refeq{vert-hor}(ii)).

 \subsection{Quasi-combies and w-membranes}  \label{ssec:combi-membr}

We assume that the zonogon $Z':=Z(n,2)$ is generated by the vectors
$\xi_i=\pi_\eps(\theta_i)$, $i=1,\ldots,n$, where the $\theta_i$ are as
in~\refeq{cyc_conf}; then the $\xi_i$ satisfy~\refeq{strict_xi}. Speaking of
combies and etc., we use terminology and notation as in Sect~\SSEC{combi}.

A \emph{quasi-combi} on $Z'$ is defined in the same way as a combi, with the
only difference that the requirement that for any lens $\lambda$, the lower
boundary $L_\lambda$, as well as the upper boundary $U_\lambda$, has at least
two edges is now withdrawn; so one of $L_\lambda$ and $U_\lambda$ is allowed to
have only one edge. When all vertices of $\lambda$ are contained in
$L_\lambda$, and therefore $U_\lambda$ has a unique edge, namely,
$(\ell_\lambda,r_\lambda)$, we say that $\lambda$ is a \emph{lower semi-lens}.
Symmetrically, when the set $V_\lambda$ of vertices of $\lambda$ belongs to
$U_\lambda$, $\lambda$ is called an \emph{upper semi-lens}. An important
special case of a semi-lens $\lambda$ is a (lower of upper) triangle.

We refer to the $\Delta$- and $\nabla$-tiles of a quasi-combi $K$ as
\emph{vertical} ones, and to the lenses and semi-lenses in it as
\emph{horizontal} ones. This is justified by the fact that all vertices $A$ of
a horizontal tile have the same size, or, let us say, lie in the same
\emph{level} $h=|A|$, whereas a vertical tile has vertices in two levels.


A quasi-combi is called \emph{fully triangulated} if all its tiles are
triangles. An immediate observation is that
  \begin{numitem1} \label{eq:fully_triang}
$\pi_\eps$ maps any w-membrane $M$ of $\Qfrag$ to a fully triangulated
quasi-combi (regarding $M$ as a 2-dimensional complex).
  \end{numitem1}

In what follows we will liberally identify $M$ with $\pi_\eps(M)$ and speak of
a w-membrane as a quasi-combi. A property converse to~\refeq{fully_triang}, in
a sense, is valid in a more general situation. Before stating it, we introduce
four simple operations on a quasi-combi $K$.
\smallskip

\noindent\textbf{(S) Splitting a horizontal tile.} For chosen a lens $\lambda$
of $K$ and non-adjacent vertices $u,v$ in $L_\lambda$ or in $U_\lambda$, the
operation cuts $\lambda$ into two pieces (either one lens and one semi-lens or
two semi-lenses) by connecting $u,v$ by the line-segment $[u,v]$. When $\lambda$ is
a lower (upper) semi-lens and $u,v$ is a pair of non-adjacent vertices in
$L_\lambda$ (resp. $U_\lambda$), the operation acts similarly.
 \smallskip

\noindent\textbf{(M) Merging two horizontal tiles.} Suppose that $\lambda'$ and
$\lambda''$, which are either two semi-lenses or one lens and one semi-lens,
have a common edge $e$ that is the longest edge of at least one of them,
$\lambda'$ say, i.e., $e=(\ell_{\lambda'},r_{\lambda'})$. The operation merges
$\lambda',\lambda''$ into one piece $\lambda:=\lambda'\cup\lambda''$.
  \smallskip

One can see that both operations result in correct quasi-combies. Two examples
are illustrated in the picture.

  \vspace{-0.1cm}
\begin{center}
\includegraphics[scale=1]{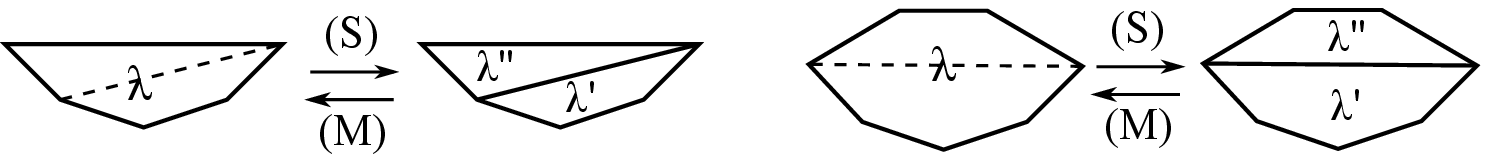}
\end{center}
\vspace{-0.1cm}

The next two operations involve semi-lenses and vertical triangles and
resemble, to some extent, tetrahedral flips in w-membranes. Here by a
\emph{lower} (\emph{upper}) \emph{fan} in a quasi-combi $K$ we mean a sequence
of $\nabla$-tiles $\nabla_r=\nabla(X|i_{r-1}i_r)$ (resp. $\Delta$-tiles
$\Delta_r=\Delta(Y|i_{r-1}i_r)$, $r=1,\ldots,p$, where $i_0<\cdots<i_p$ (resp.
$i_0>\cdots >i_p$); i.e., these triangles have the same bottom vertex $X$ (resp.
the same top vertex $Y$) and two consecutive triangles share a vertical edge.
  \smallskip

\noindent\textbf{(E) Eliminating a semi-lens.} Suppose that the longest edge
$e=(\ell_\lambda,r_\lambda)$ of a lower semi-lens $\lambda$ belongs to a
$\Delta$-tile $\Delta=\Delta(Y|ji)$ ($j>i$). Then $e$ is the base edge
$(Y-j,Y-i)$ of $\Delta$, and $\lambda$ has type $ij$ and the upper root just at
$Y$. The operation of eliminating $\lambda$ replaces $\lambda$ and $\Delta$ by
the corresponding upper fan $(\Delta_r\colon r=1,\ldots,p)$, where each
$\Delta_r$ has the top vertex $Y$ and its base edge is $r$-th edge in
$L_\lambda$. Symmetrically, if an upper semi-lens $\lambda$ and a $\nabla$-tile
$\nabla$ share an edge $e$ (which is the longest edge of $\lambda$ and the base
edge of $\nabla$), then the operation replaces $\lambda$ and $\nabla$ by the
corresponding lower fan $(\nabla_r\colon r=1,\ldots,p)$, where the base edge of
$\nabla_r$ is $r$-th edge in $U_\lambda$.
  \smallskip

\noindent\textbf{(C) Creating a semi-lens.} This operation is converse to~(E).
It deals with a lower or upper fan of vertical triangles and replaces them by
the corresponding pair consisting of either an upper semi-lens and a
$\nabla$-tile, or a lower semi-lens and a $\Delta$-tile.
  \smallskip

Again, it is easy to check that (E) and (C) result in correct quasi-combies.
These operations are illustrated in the picture (where $p=3$).

 \vspace{0cm}
\begin{center}
\includegraphics[scale=0.7]{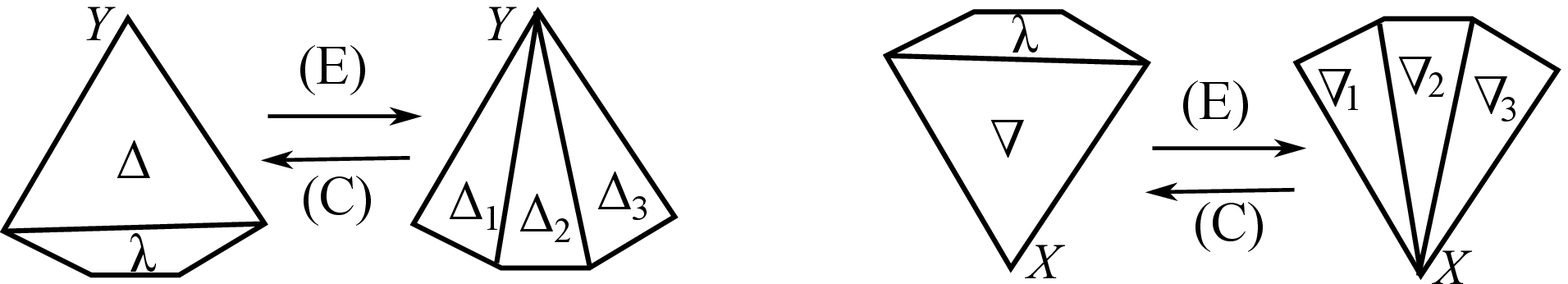}
\end{center}
\vspace{0cm}

Now, given a quasi-combi $K$, we consider the set $\Omega(K)$ of all
quasi-combies $K'$ on $Z'$ with the same spectrum $V_K$, called the
\emph{escort} of $K$ (note that when $K$ matches a w-membrane $K$, $\Omega(K)$
can be larger than $\Escr(M)$). We observe the following
  \begin{lemma} \label{lm:OrbitK}
{\rm(i)} ~$\Omega(K)$ contains exactly one combi. {\rm(ii)} ~$\Omega(K)$ is the set of
quasi-combies that can be obtained from $K$ by use of operations
(S),(M),(E),(C). In particular, $V_K$ is a maximal w-collection in $2^{[n]}$.
  \end{lemma}
   \begin{proof}
Choosing an arbitrary quasi-combi $K'\in\Omega(K)$ and applying to $K'$ a
series of operations~(M) and (E), one can produce $K^\ast$ having no
semi-lenses at all (since each application of (M) or (E) decreases the number
of semi-lenses). Therefore, $K^\ast$ is a combi with $V_{K^\ast}=V_K=:S$.
Moreover, $K^\ast$ is the unique combi with the given spectrum $S$, by
Theorem~\ref{tm:combi} (see also~\cite[Th.~3.5]{DKK3}). This gives~(i). Now (i)
implies~(ii) (since any $K'\in \Omega(K)$ can be obtained from $K^\ast$ using
(S) and (C), which are converse to (M) and (E)).
   \end{proof}

As a consequence of~\refeq{fully_triang} and Lemma~\ref{lm:OrbitK}, we obtain

  \begin{corollary} \label{cor:wmembr-maxwc}
The spectrum of any w-membrane is a maximal w-collection in $2^{[n]}$.
 \end{corollary}

\noindent\textbf{Definition.} A quasi-combi $K$ is called \emph{compatible}
with a cubillage $Q$ if each edge of $K$ is (the image by $\pi_\eps$ of) an edge
of $\Qfrag$. (In particular, $V_K\subset V_Q$.)

  \begin{prop} \label{pr:compat_quasi}
Let $K$ be a quasi-combi on $Z'=Z(n,2)$ compatible with a cubillage $Q$ on
$Z(n,3)$. Then the horizontal tiles (lenses and semi-lenses) of $K$ can be
triangulated so as to turn $K$ into (the image by $\pi_\eps$ of) a
w-membrane in $\Qfrag$.
  \end{prop}
  \begin{proof}
Let $\tau$ be a $\Delta$- or $\nabla$-tile in $K$; then the edges of $\tau$
belong to $\Qfrag$. Arguing as in the proof of
Proposition~\ref{pr:edge_rh_cube} (using induction on $n$ and considering the
$n$-contraction of $Q$ and its fragmentation), one can show that $\tau$ is a
face (a vertical triangle) of $\Qfrag$. Now consider a lens
or semi-lens $\lambda$ of $K$ lying in level $h$, say. Since all edges of
$\lambda$ belong to $\Qfrag$, the polygon $\lambda$ must be subdivided into a
set of triangles in the section of $\Qfrag$ by the plane $z=h$. Combining such
sets and vertical triangles $\tau$ as above, we obtain a disk bijective to $Z'$
by $\pi_\eps$, yielding a w-membrane $M$ in $\Qfrag$ with $V_K\subset V_M$. Now
the fact that both $V_M$ and $V_K$ are maximal w-collections implies $V_K=V_M$,
and the result follows.
  \end{proof}

Let us say that a quasi-combi $K$ is \emph{non-expensive} if all semi-lenses in
it are triangles and there is no semi-lens $\lambda$ whose longest edge
$(\ell_\lambda,r_{\lambda})$ is simultaneously either an edge of a lens or the
longest edge of another semi-lens. In particular, any combi is non-expensive.

Note that~\refeq{fully_triang} and Proposition~\ref{pr:compat_quasi} imply that
each w-membrane $M$ one-to-one corresponds (via $\pi_\eps$) to a fully
triangulated quasi-combi compatible with $Q$ and having the same spectrum
$V_M$. One more correspondence following from Proposition~\ref{pr:compat_quasi}
concerns non-expensive quasi-combies.

 \begin{corollary} \label{cor:non-expens}
Each w-membrane $M$ one-to-one corresponds to a non-expensive quasi-combi $K$
compatible with $Q$ and such that $V_K=V_M$. Non-expensive quasi-combies with
the same escort have the same set of lenses.
  \end{corollary}

Indeed, for a non-expensive quasi-combi $K$, the corresponding w-membrane $M$
is obtained by subdividing each lens of $K$ into triangles of $\Qfrag$. We also
use the fact that each application of~(E) matches a tetrahedral flip in the
corresponding w-membrane (since each semi-lens is a triangle), and a series of
such operations results in a combi with the same set of lenses.

A sharper version of above results is stated by weakening the requirement of
compatibility.

  \begin{theorem} \label{tm:combi-wmembr}
For each maximal by size w-collection $W$ contained in the spectrum $V_Q$ of a
cubillage $Q$, there exists a w-membrane $M$ in $\Qfrag$ with $V_M=W$.
  \end{theorem}
  \begin{proof}
Let $K$ be the combi with $V_K=W$. In light of reasonings in the proof of
Proposition~\ref{pr:compat_quasi}, it suffices to show that
  \begin{numitem1} \label{eq:vert_trian}
each vertical triangle $\tau$ of $K$ is extended to a rhombus of $Q$ (and
therefore $\tau$ is a face of $\Qfrag$).
  \end{numitem1}

To see this, we rely on the following fact (which is interesting in its own
right).
  \medskip

\noindent\textbf{Claim.} \emph{Let a set $Y\subset[n]$ be chord separated from
each of $X,X1,Xn$ for some $X\subseteq [n]-\{1,n\}$. Then $Y$ is chord
separated from $X1n$ as well.}
  \medskip

\noindent\textbf{Proof of the Claim.} Let $1,\ldots,n$ be disposed in this
order on a circumference $O$. Let $Y':=Y-X$ and $X':=X-Y$. One may assume that
$1,n\notin Y'$ (otherwise the chord separation of $Y$ and $X1n$ immediately
follows from that of $Y,X,X1,Xn$).

If $Y$ and $X1n$ are not chord separated, then there are elements $x,x'\in
X'1n$ and $y,y'\in Y'$ such that the corresponding chords $e=[x,x']$ and
$e'=[y,y']$ ``cross'' each other. Then $\{x,x'\}\ne\{1,n\}$ (since $1,n$ are
neighboring in $O$). So one may assume that $x\in X'$ (and $x'\in X'1n$). But
in each possible case $(x'\in X'$, $x'=1$ or $x'=n$), the chord $e$ crossing
$e'$ connects two elements of either $X'$ or $X'1$ or $X'n$; a contradiction.
\hfill \qed
  \medskip

Now consider a $\nabla$-tile $\nabla=\nabla(X|ij)$ of $K$ (having the vertices
$X,Xi,Xj$ with $i<j$). If $\{i,j\}=\{1,n\}$, then, by the Claim (and
Theorem~\ref{tm:galash}), $Xij$ is chord separated from all vertices of $Q$,
and the maximality of $V_Q$ implies that $Xij$ is a vertex of $Q$ as well.
Hence, by Proposition~\ref{pr:edge_rh_cube}(ii), $Q$ contains the rhombus
$\rho(X|ij)$, as required.

So we may assume that at least one of $j<n$ and $1<i$ takes place. Assuming the
former, we use induction on $n$ and argue as follows.

Let $Q'$ be the $n$-contraction of $Q$, and $M$ the s-membrane in $Q'$ that is
the image of the $n$-pie in $Q$ (for definitions, see Sect.~\SSEC{pies}).
Besides $Q'$, we need to consider the reduced set $W':=\{A\subseteq[n-1]\colon
A$ or $An$ or both belong to $W\}$. Then $W'$ is a maximal w-collection in
$2^{[n-1]}$, and as is shown in~\cite{DKK3},
  \begin{numitem1} \label{eq:reduc_tau}
if $\tau$ is a vertical triangle of $K$ having type $ij$ with $j<n$, if $A,B,C$
are the vertices of $\tau$, and if $K'$ is the combi on $Z(n-1,2)$ with
$V_{K'}=W'$, then $K'$ has a vertical triangle with the vertices
$A-n,\,B-n,\,C-n$.
  \end{numitem1}

Now consider two: $n\notin X$ and $n\in X$.

If $n\notin X$, then $X,Xi,Xj$ are vertices of $Q'$ and simultaneously
vertices of the reduced combi $K'$. By~\refeq{reduc_tau}, $K'$ has the tile
$\nabla'=\nabla(X|ij)$. By indiction, the vertices of $\nabla'$ are extended to
a rhombus $\rho'$ of $Q'$. This $\rho'$ is lifted to $Q$, as required.

If $n\in X$, then $Q'$ and $K'$ have vertices $X',X'i,X'j$ for $X':=X-n$, ~$K'$
has the triangle $\nabla'=\nabla(X'|ij)$ (by~\refeq{reduc_tau}), the vertices
of $\nabla'$ are extended to a rhombus $\rho'$ of $Q'$, and $\rho'$ is lifted
to the desired rhombus $\rho(X|ij)$ in $Q$.

The case of a $\Delta$-tile $\Delta=\Delta(Y|ji)$ of $K$ with $i<j<n$ is symmetric.

Finally, if $1<i<j=n$, we act in a similar fashion, but applying to $Q$ the
1-contraction operation, rather than the $n$-contraction one (this is just the
place where we use the 1-contraction mentioned in Sect.~\SSEC{pies}); the
details are left to the reader.

This completes the proof of the theorem.
  \end{proof}


\section{Extending a combi to a cubillage} \label{sec:embed_combi}

The purpose of this section is to explain how to efficiently extend a fixed
maximal w-collection in $2^{[n]}$ to a maximal c-collection, working with their
geometric interpretations: combies and cubillages. Our construction will imply the
following

  \begin{theorem} \label{tm:embed_combi}
Given a maximal weakly separated collection $W\subset 2^{[n]}$, one can find,
in polynomial time, a maximal chord separated collection $C\subset 2^{[n]}$
including $W$.
  \end{theorem}
  \begin{proof}
It is convenient to work with an arbitrary fully triangulated quasi-combi $K$
with $V_K=W$. The goal is to construct a cubillage $Q$ on $Z=Z(n,3)$ whose
fragmentation $\Qfrag$ contains $K$ as a w-membrane. (Note that it is routine
to construct the (unique) combi with the spectrum $W$ (see~\cite{DKK3} for
details), and to form $K$, we subdivide each lens of the combi into the pair of upper and
lower semi-lenses and then triangulate them arbitrarily. The resulting
cubillage $Q$ will depend on the choice of such triangulations.)

We start with properly embedding $K$ into the ``empty'' zonotope $Z$, and our
method consists of two phases. At the first (second) phase, we construct a
partial fragmentation $F^-$ (resp. $F^+$) consisting of $\nabla$-, $\square$-,
and $\Delta$-fragments of some cubes $\zeta(X|ijk)$ (where, as usual, $i<j<k$
and $X\subseteq [n]-\{i,j,k\}$) filling the region $Z^-(K)$ of $Z$ between
$\Zfr$ and $K$ (resp. the region $Z^+(K)$ between $K$ and $\Zrear$). (For
definitions, see Sect.~\SSEC{fragment}.) Then $F:=F^+ \cup F^-$ is a
subdivision of $Z$ into such fragments, and it is not difficult to realize that
$F$ is just the fragmentation $\Qfrag$ of some cubillage $Q$; so $\Qfrag$ is as
required for the given $K$.

Next we describe the first phase. At each step in it, we deal with one more
w-membrane $M$ such that
  \begin{itemize}
\item[$(\ast)$]
$M$ lies entirely in $Z^-(K)$, and there is a partial fragmentation $F'$
filling the region $Z(M,K)$ between $M$ and $K$ (i.e., $F'$ is a
subdivision of $Z(M,K)$ into $\nabla$-, $\square$-, and $\Delta$-fragments).
  \end{itemize}

If $M$ (regarded as a fully triangulated quasi-combi) has no horizontal
triangle (semi-lens), then $M$ is, in essence, a rhombus tiling in which
each rhombus $\rho(X|ij)$ is cut into two vertical triangles, namely,
$\nabla(X|ij)$ and $\Delta(Xij|ji)$. So $M$ can be identified with the
corresponding s-membrane, and we can construct a partial cubillage $Q'$ filling
the region $Z^-(M)$ (between $\Zfr$ and $M$) by acting as in
Sect.~\SSEC{memb-to-cube}. Combining $Q'$ and $F'$, we obtain the desired
fragmentation $F^-$ filling $Z^-(K)$.

Now assume that $M$ has at least one semi-lens, and let $h$ be minimum so that
the set $\Lambda$ of semi-lenses in the level $z=h$ is nonempty. Choose
$\lambda\in\Lambda$ such that no edge in its lower boundary $L_\lambda$ belongs
to another semi-lens. (The existence of such a $\lambda$ is provided by the
acyclicity of the directed graph whose vertices are the elements of $\Lambda$
and whose edges are the pairs $(\lambda,\lambda')$ such that $U_\lambda$ and
$L_{\lambda'}$ share an edge, which follows, e.g., from
Lemma~\ref{lm:acyclic}.) Two cases are possible.
  \medskip

\emph{Case 1}: $\lambda$ is an upper triangle, i.e., $L_\lambda$ consists of a
single edge, namely, $e=(\ell_\lambda,r_\lambda)$. Let $U_\lambda$ have vertices
$Xi=\ell_\lambda$, $Xj$ and $X_k=r_\lambda$ ($i<j<k$). Then $e$ belongs to a
$\nabla$-tile in $M$, namely, $\nabla=\nabla(X|ik)$. Form the $\nabla$-fragment
$\tau=\zeta^\nabla(X|ijk)$ (the lower tetrahedron with the vertices
$X,Xi,Xj,Xk)$. We add $\tau$ to $F'$ and accordingly make the lowering flip in
$M$ using $\tau$ (which replaces the triangles $\lambda,\nabla$ forming
$\taurear_\eps$ by $\nabla(X|ij)$ and $\nabla(X|jk)$ forming $\taufr_\eps$; see
Sect.~\SSEC{wmembran}). The new $M$ is a correct fully triangulated quasi-combi
(embedded as a w-membrane in $Z$), which is closer to $\Zfr$.
  \medskip

\emph{Case 2}: $\lambda$ is a lower triangle. Then $L_\lambda$ consists of two
edges: $e=(\ell_\lambda=Y-k,Y-j)$ and $e'=(Y-j,Y-i=r_\lambda)$, where $i<j<k$.
Also by the choice of $h$, the edges $e,e'$ belong to $\nabla$-tiles of $M$,
namely, those of the form $\nabla=\nabla(X|jk)$ and $\nabla'=\nabla(X'|ij)$,
respectively, where $X:=Y-\{j,k\}$ and $X':=Y-\{i,j\}$. See the left fragment
of the picture.

   \vspace{-0.2cm}
\begin{center}
\includegraphics[scale=0.85]{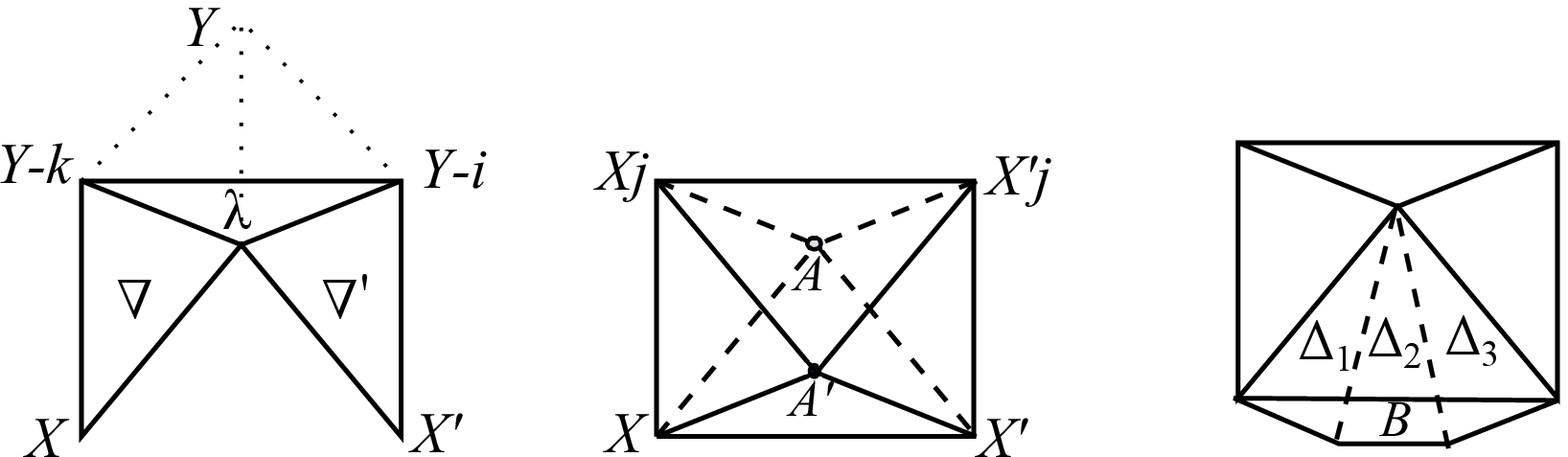}
\end{center}
\vspace{0cm}

Let $A:=Y-j$ ($=Xk=X'i$). Note that the ``sector'' between the edges $(X,A)$
and $(X',A)$ is filled by an upper fan $(\Delta_1,\ldots,\Delta_p)$, where
$\Delta_r=\Delta(A|i_{r-1}i_r)$ and $k=i_0>i_1\cdots>i_p=i$ (cf.
Sect~\SSEC{wmembran}). Consider two possibilities.
  \smallskip

\emph{Subcase 2a}: $p=1$, i.e., the fan consists of only one tile, namely,
$\Delta=\Delta(A|ki)$. Observe that the vertices $X,X',Xj,X'j,A$ belong
to an octahedron, namely, $\tau=\zeta^\square(\tilde X|ijk)$, where $\tilde X$
denotes $X-i=X'-k$. Moreover, the triangles $\lambda,\Delta,\nabla,\nabla'$
form the rear side of $\tau$. We add $\tau$ to $F'$ and accordingly make the
octahedral flip in $M$ using $\tau$, which replaces $\taurear_\eps$ by the front side
$\taufr_\eps$ formed by four triangles shared the new vertex $A':=\tilde Xj$.
(See the middle fragment of the above picture where the new triangles are
indicated by solid lines.) The new $M$ is again a correct w-membrane
closer to $\Zfr$. Note that under the flip, the semi-lens $\lambda$ is replaced
by an upper semi-lens $\lambda'$ in level $h-1$ (this $\lambda'$ has the
longest edge $(X,X')$ and the top $A'$).
  \smallskip

\emph{Subcase 2b}: $p>1$. Then $X$ and $X'$ are connected in $M$ by the path
$P$ that passes the vertices $X=A-i_0,A-i_1,\ldots, A-i_p=X'$. We make two
transformations. First we connect $X$ and $X'$ by line-segment $\tilde e$.
Since $\tilde e$ lies in the region $Z^-(M)$ (in view of~\refeq{cyc_conf}), so
does the entire truncated polyhedral cone $\Sigma$ with the top vertex $A$ and
the base polygon $B$ bounded by $P\cup\tilde e$. See the right fragment of the
above picture (where $p=3$). We subdivide $B$ into $p-1$ triangles
$\sigma_1,\ldots,\sigma_{p-1}$ (having vertices on $P$) and extend each
$\sigma_r$ to tetrahedron $\tau_r$ with the top $A$. These
$\tau_1,\ldots,\tau_{p-1}$ subdivide $\Sigma$ into $\Delta$-fragments (each
being of the form $\zeta^\Delta(A|i_\alpha i_\beta i_\gamma)$ for some $0\le
\alpha<\beta<\gamma\le p$). Observe that the rear side $\Sigma^{\rm re}_\eps$
of $\Sigma$ is formed by the fan $(\Delta_1,\ldots,\Delta_p)$, whereas
$\Sigma^{\rm fr}_\eps$ consists of the lower horizontal triangles
$\sigma_1,\ldots,\sigma_{p-1}$ plus the vertical triangle with the top $A$ and
the base $\tilde e$, denoted as $\tilde \Delta$.

We add the fragments $\tau_1,\ldots,\tau_{p-1}$ to $F'$ and accordingly update
$M$ by replacing the triangles of $\Sigma^{\rm re}_\eps$ by the ones of
$\Sigma^{\rm fr}_\eps$ (as though making $p-1$ lowering tetrahedral flips). The
new w-membrane has the upper fan at $A$ consisting of a unique $\Delta$-tile,
namely, $\tilde \Delta$, and now we make the second transformation, by applying
the octahedral flip as in Subcase~2a (involving the triangles
$\lambda,\tilde\Delta,\nabla,\nabla'$ on the same vertices $X,X',Xj,X'j,A$).

Doing so, we eventually get rid of semi-lenses in the current $M$; so $M$
becomes an s-membrane in essence, which enables us to extend the current $F'$
to the desired fragmentation $F^-$ filling $Z^-(K)$ (by acting as in
Sect.~\SSEC{memb-to-cube}).

At the second phase, we act ``symmetrically'', starting with $M:=K$ and moving
toward $\Zrear$, in order to obtain a fragmentation $F^+$ filling $Z^+(K)$.
Then $F^-\cup F^+$ is as required, and the theorem follows.
  \end{proof}



\begin{thebibliography}{99}

 %
\bibitem{DKK1} V.I.~Danilov, A.V.~Karzanov and G.A.~Koshevoy, Pl\"ucker environments,
wiring and tiling diagrams, and weakly separated set-systems,
\textsl{Adv.~Math.} \textbf{224} (2010) 1--44.
 %
\bibitem{DKK2} V.I.~Danilov, A.V.~Karzanov and G.A.~Koshevoy, On maximal
weakly separated set-systems, \textsl{J.~of Algebraic Combinatorics}
\textbf{32} (2010) 497--531.
 %
\bibitem{DKK3} V.I.~Danilov, A.V.~Karzanov and G.A.~Koshevoy,
Combined tilings and the purity phenomenon on separated set-systems,
\textsl{Selecta Math. New Ser.} \textbf{23} (2017) 1175--1203.
%
\bibitem{FW} S.~Felsner and H.~Weil, A theorem on higher Bruhat orders,
\textsl{Discrete Comput.~Geom.} \textbf{23} (2000) 121--127.
 %
\bibitem{gal} P.~Galashin, Plabic graphs and zonotopal tilings,
\textsl{ArXiv}:1611.00492[math.CO], 2016.
 %
\bibitem{GP} P.~Galashin and A.~Postnikov, Purity and separation for oriented
matroids, \textsl{ArXiv}:1708.01329[math.CO], 2017.
 %
\bibitem{LZ} B.~Leclerc and A.~Zelevinsky: Quasicommuting families of
quantum Pl\"ucker coordinates, {\textsl Amer. Math. Soc. Trans., Ser.~2} ~{\bf
181} (1998) 85--108.
  %
\bibitem{MS} Yu.~Manin and V.Schechtman, Arrangements of hyperplanes, higher
braid groups and higher Bruhat orders, in: \textsl{Algebraic Number Theory --
in Honour of K.~Iwasawa}, \textsl{Advances Studies in Pure Math.} \textbf{17},
Academic Press, NY, 1989, pp.~289--308.
 %
\bibitem{OPS} S.~Oh, A.~Postnikov, and D.E.~Speyer, Weak separation and plabic
graphs, {\textsl ArXiv}:1109.4434[math.CO], 2011.
 %
\bibitem{OS} S.~Oh and D.E.~Speyer, Links in the complex of weakly separated collections,
{\textsl ArXiv}:1405.5191[math.CO], May~2014.
  %
\bibitem{VK} V.~Voevodskij and M.~Kapranov, Free $n$-category generated by a
cube, oriented matroids, and higher Bruhat orders, \textsl{Functional Anal.
Appl.} \textbf{2} (1991) 50--52.
 %
\bibitem{zieg} G.~Ziegler, Higher Bruhat orders and cyclic hyperplane
arrangements, \textsl{Topology} \textbf{32} (1993) 259--279.

\end{thebibliography}
 \end{document}